\documentclass[a4paper,11pt]{article}

\usepackage[latin1]{inputenc}

\usepackage[english]{babel}

\usepackage{amssymb}
\usepackage{latexsym}

\usepackage{amsfonts}

\usepackage{amsmath}

\usepackage{bm}

\usepackage{bbm}

\usepackage{mathrsfs}

\usepackage{enumerate}

\usepackage[T1]{fontenc}
\usepackage{ae}

\usepackage{graphicx}

\usepackage[english]{varioref}

\newcommand{\N}{\mathbbm{N}}
\newcommand{\C}{\mathbbm{C}}

\newcommand{\R}{\mathbbm{R}}
\newcommand{\Z}{\mathbbm{Z}}
\newcommand {\V}{\mathbbm{V}}

\usepackage[amsmath,thmmarks]{ntheorem}

\theoremseparator{.}

\newtheorem{thm}{Theorem}
\newtheorem{prop}[thm]{Proposition}
\newtheorem{lemma}[thm]{Lemma}

\theorembodyfont{\normalfont}

\theoremheaderfont{\normalfont\itshape}
\theorembodyfont{\normalfont}

\newtheorem{remark}[thm]{Remark}

\theoremstyle{nonumberplain}
\theoremheaderfont{\normalfont\itshape}
\theorembodyfont{\normalfont}
\theoremsymbol{$\square$}
\newtheorem{proof}{Proof}

\begin{document}

\title{Conformal symmetries in geometry\\ and harmonic analysis\\
Introductory lectures at IHP Paris, January  2025\\
trisemester program on
\\Representation Theory \\and Noncommutative Geometry}
\author{Bent \O{}rsted, Aarhus University}
\date{} 

\maketitle  
%
%
%


In these lectures we wish to give a gentle introduction and overview of two
aspects of conformal symmetry, ususally not treated together; one is that
of conformal differential geometry on a Riemannian manifold, in particular
invariants related to heat equations for elliptic operators, and the other
that of representations of the conformal group. Both are applications of the
conformal covariance of differential operators such as the Yamabe operator.
Specifically we shall see how representation theory can be applied to the
study of extremal properties of determinants of some elliptic operators,
and similarly how geometry can be applied to find the explicit branching law
for the minimal representation of $O(p,q)$ with respect to some symmetric
subgroups. In effect, we give three models of the representation, an elliptic,
a hyperbolic, and a parabolic.

\section{Overview and motivation}

The theory of Lie groups was from the beginning closely connected to geometry and to solving differential
equations; Sophus Lie himself was motivated by geometries such as those of Bernhard Riemann and also
projective geometry, and he saw the analogy to Galois theory in studying the symmetries of
differential equations. See the paper with historical notes by S. Helgason \cite{he2}, where it is
explained that Lie had heard via Sylow about the discoveries of Galois of the connection between the
structure of the permutations (the Galois group)
of the roots of a polynomial equation, and the solvability (or
non--solvability) of this equation by radicals. Lie had the idea of doing something similar for
differential equations. Thus his theory of transformation groups was relevant - and had already been
used for geometric problems in projective geometry. One might say that for Lie the
symmetries (of an equation or a geometry) was the basic object, whereas for Riemann the
basic object was the geometry itself, its geodesics, curvature and so on.
Both have had influence on theoretical physics, Riemann in connection with
general relativity through his study of metric tensors and curvatures, essentially in a way independent
of choices of coordinates, and Lie for the roles of symmetry, usually in the form of a group of
transformations of space (or physical objects in space) - 
sometimes by infinitesimal transformations, i.e. vector fields -
this with deep applications in both classical mehanics and
in the later quantum mechanics.
From a modern mathematical point of view, we have the similar two aspects of conformal
geometry, namely let $(M, g)$ be s smooth manifold with a smooth Riemannian metric tensor $g$, then
\begin{itemize} 
\item a {\it conformal change} of metric means replacing $g$ by $\overline{g} = e^{2\omega}g$
with $\omega \in C^{\infty}(M, \R)$, a smooth real function on $M$.
\item a {\it conformal transformation} means a (perhaps only local) diffeomorphism
$\Phi: M \mapsto M$ with $\Phi^*g = e^{2\omega}g$, for some (depending on $\Phi$)
$\omega \in C^{\infty}(M, \R)$, 
i.e. the pull-back of the metric
is conformal to itself.
\end{itemize}
In both cases the geometric meaning is that the angle between two smooth curves
at a point of intersection is preserved. The metric $g$ gives an inner product in the tangent space
and hence the angle between the two tangent vectors, and this is preserved by multiplying
the inner product by a positive number. Actually Riemann himself gave such an example (one of the
few formulas in the paper) in his thesis about the the foundations of geometry, namely in $\R^n$
$$ds^2 = \frac{ds_0^2}{(1 + K |x|^2)^2}$$
where the numerator and $|x|^ 2$ is the ususal Euclidian metric, and the new metric is of constant
curvature, either positive, zero, or negative, depending on the sign of the constant $K$.

Conformal geometry, as for example in \cite{bra_1}, \cite{bran_ors1}, \cite{bran_ors2} has as objective to study quantities that 
are invariant, or change in a predictable way, under conformal changes of metric; these could
in particular be functionals $F$ on the space of all metrics $g$ on a fixed $M$, invariant
under conformal changes: $F(e^{2\omega}g) = F(g)$ for all $g$ and $\omega$.
Another important topic is to find differential operators $L = L(g)$ canonically associated the Riemannian structure, say on functions or on sections of natural bundles; the operators 
should exhibit {\it conformal symmetry}, namely
$$L(\overline{g}) = L(e^{2\omega}g) = e^{-b\omega} L(g) e^{a\omega}$$
for some real constants $a, b \in \R$. Here the right hand side is the composition of multiplication
operators and the differential operator $L(g)$. The numbers  $(a, b)$ are called the conformal bidegrees of
$L$. We shall also refer to such operators as conformally covariant differential operators, and
they will be natural in the sense of being constructed via the metric tensor $g$, i.e. via the geometry
on $M$. A prominent example is the {\it Yamabe operator}
$$Y = \Delta + \frac{n-2}{4(n-1)} K$$
with $K = g^{ij}r_{ij}$ the scalar curvature (see below for the index notation,
and the notions of curvatures and connection $\nabla$ - the Ricci curvature as a
tensor in local coordinates is here denoted $r_{ij}$), $n$ the dimension
of $M$
and $\Delta = - g^{ij}\nabla_i \nabla_j$
the (non-negative) Laplace operator on functions, sometimes also written $\Delta = d^* d$.
The Yamabe operator is conformally covariant with bidegrees $a = (n-2)/2, \, b =(n+2)/2$.
This is the starting point of our treatment, namely there will be two kinds of consequences of this
conformal symmetry, following the philosophical principle that a symmetry of a system gives
rise to conserved quantities: 

(1) In conformal differential geometry we shall find some 
conformally invariant functionals
$F = F(g)$ depending on the Riemannian metric $g$ on a fixed compact manifold $M$. They arise
primarily through a study of the heat equation
$$\frac{\partial u(t,x)}{\partial  t} = Y u(t,x)$$
for a function $u = u(t,x), t \in \R, x \in M$ with $Y$ the Yamabe operator on $M$ (or a similar
conformally covariant elliptic differential operator, acting in the $x$ variable). Using the theory of pseudodifferential
operators and Sobolev spaces, one constructs a fundamental solution with an asymptotic
expansion for small $t$, and coefficients reflecting the geometry of $(M, g)$. The conformal
symmetry of $Y$ gives information about the change of these coefficients under conformal
change of metric.

(2) In the representation theory of the indefinite orthogonal group $G = O(p,q)$,
viewed as the conformal group of the product of spheres $S^{p-1} \times S^{q-1}$ with
the natural indefinite Riemannian metric (positive on the first factor, negative on the second), we shall construct its miminal unitary representation in the kernel of the Yamabe operator.
At the same time we get models that are convenient for finding the
branching law, i.e. the explicit restriction of this representation to some symmetric subgroups of $G$.
Here it is crucial that the Yamabe equation $Y \varphi = 0$ is invariant under conformal changes
of coordinates, so we may solve it in any such system; in effect it corresponds to the well-known
separation of variables in solving partial differential equations, where the separations comes from
breaking the symmetries - and not surprisingly gives the symmetry-breaking operators and the branching
law. Just as in the case of the metaplectic (harmonic/oscillator) representation there are interesting
aspects of the various realizations in concrete Hilbert spaces, and we shall give in fact three such,
namely corresponding to conic sections of elliptic, hyperbolic, and parabolic type. The last one is
closely connected with Fourier analysis in $\R^n$ with an indefinite symmetric bilinear form.
From the point of view of the philosophical principle of symmetry, the conserved quantities
are now subspaces of solutions to the Yamabe equation. 

In this note we are trying to present some ideas centered around the notion of conformal symmetry;
some details and (several) proofs are left out, but hopefully the references given will allow
a fuller understanding. It is a topic which illustrates how Lie theory and geometry are
interconnected - and can (hopefully) serve as an inspiration to other similar geometric
situations. One lesson is that working on a homogeneous space $M = G/P$ for a Lie
group $G$ and a closed subgroup $P$ (typically a parabolic subgroup
in $G$, a semisimple Lie group), one can often 
learn about geometries on deformations
of $M$, and $G/P$ in this way serves as a "flat" model for more general "curved" (parabolic) geometries $M.$
Our primary example is the standard $n$-sphere $S^n$ with its group of conformal transformations
by $G =  SO(n+1,1)$ as a model for more general conformal geometry in $n$ dimensions.

{\bf Acknowledgement.} Thanks are due to the referee for a careful reading and suggestions
to improve this essay. 

\section{Conformal differential geometry}
For the basic notions of connections and curvatures see \cite{he}. We denote by $T_x = TM_x$ the tangent space to the manifold $M$ at the point $x \in M$ and by $T^*_x = TM_x^*$ the dual co-tangent
space; by $T$ denote the corresponding smooth vector bundle and $T^*$ he dual bundle. Vector
fields $X, Y, \dots$ are smooth sections of $T$, i.e. $X, Y \in C^{\infty}(T)$. We may form tensor
products of $T$ and $T^*$, giving rise to tensor fields as (always smooth) sections. A metric tensor
$g$ is then a section of the symmetric tensors of rank two, i.e. $g \in C^{\infty}(S^2 T^*)$ with the
extra condition of defining a symmetric non-degenerate bilinear form in each tangent space $T_x$; 
for a Riemannian positive metric this should be positive-definite, but may also have a signature - for
example in general relativity usually of Lorentz signature (positive for one time-like direction, and
in the remaining space-like directions negative). Given a metric there is a unique connection
$$\nabla: C^{\infty}(T) \mapsto C^{\infty}(T^* \otimes T)$$
which defines the covariant derivative of a vector field $Y$ in the direction of the vector field $X$
by $\nabla_X Y = \nabla (Y)(X)$ so $\nabla_X: C^{\infty}(T) \mapsto C^{\infty}(T)$. This is
called the Levi-Civita connection and it is determined by the torsion vanishing 
$$\nabla_X Y -  \nabla_Y X - [X,Y] = 0$$
and preserving the metric in the sense that
$$(\nabla_X Y, Z) + (Y, \nabla_X Z) = X(Y,Z)$$
for all vector fields; here the bilinear form defined by the metric is $(X,Y) \in C^{\infty}(M)$.
The Lie bracket is the vector field $[X, Y] = XY - YX$, acting on functions. There is a natural extension,
via the Leibniz rule, of the connection to all tensor fields
$$\nabla: C^{\infty}(E) \mapsto C^{\infty}(T^* \otimes E)$$
with $E$ a bundle of tensors; then the metric condition above reads $\nabla g = 0$. For a general
vector bundle $E$ a connection is such a first order partial differential operator $\nabla$
satisfying $\nabla(f \phi) = df \otimes  \phi + f \nabla \phi $ for all $f \in C^{\infty}(M)$ and
all $\phi \in C^{\infty}(E)$.
In a system of
local coordinates $x = (x^1, \dots ,x^n)$ we have the coordinate frames $(\partial_i = \frac{\partial}{\partial x^i})$ resp. $(dx^i)$ for the tangent resp. the co-tangent bundles. In these the metric reads
$ds^2 = g_{ij}dx^i \circ dx^j$ (Einstein summation convention over repeated indices, and the
symmetric tensor product $v \circ w = \frac{1}{2}(v \otimes w + w \otimes v)$ for $v, w \in T^*_x$).
Also $\nabla_{\partial_i}\partial_j = \Gamma_{ij}{}^k \partial_k$ with the so-called Christoffel symbols
$\Gamma_{ij}{}^k$ (smooth functions, but not tensors).

The covariant derivative is the basic geometric quantity, and the Christoffel symbols are first-order
derivatives in the metric tensor; the Riemann curvature tensor $R$ is a collection of
second-order derivatives of the
metric, and in coordinate-free way given as
$$R(X,Y)Z = (\nabla_X \nabla_Y - \nabla_Y \nabla_X - \nabla_{[X,Y]})Z$$
for vector fields $X, Y, Z.$ At each point $x$ it defines a multilinear mapping
$R : T_x \times T_x \times T_x  \mapsto T_x$ and thus gives an element 
$R_x \in T^*_x \otimes T^*_x \otimes T^*_x \otimes T_x.$ We may identify $T_x$ and $T^*_x$
via the metric so that in the local coordinates the inverse matrix $g^{ij}$ to the matrix $g_{ij}$
gives the corresponding bilinear form in $T^*_x.$ The orthogonal group $SO(n)$ acts in each tangent space
as the defining repesentation, and $R_x$ belongs then to a representation space $V$ for $SO(n).$
This space is the sum of three irreducible representations, namely
$$V = V_{scalar} \oplus V_{Ricci_0} \oplus V_{Weyl}$$
with highest weights $(0), (2), (2,2)$ resp. i.e. the trivial, the trace-free symmetric two-tensors,
or harmonic polynomials of degree two, and the remaining tensors. The last component is called the
Weyl curvature, and the first two scalar curvature and trace-free Ricci. They are identified as
$$Ric(X, Y) = {\rm Tr}(Z \mapsto R(X, Z)Y)$$
which is a bilinear symmetric form $r_{ij}$ and its trace $K = g^{ij} r_{ij}$ the scalar
curvature. We have that $r_{ij} = R_{ikj}{}^k$, expressing the curvature tensor with these
coefficients in local coordinates. See the paper \cite{str} (and the MathSciNet review by Vanhecke) for
a discussion of this $O(n)$ representation theory, including that of $\nabla R$ and its geometric
interpretations; note that for $n \geq 4$ the vanishing of the Weyl tensor means that the metric
is conformally flat, i.e. it is conformal to the Euclidian metric. For $n = 3$ the $(2,2)$ is not there.
Curiously the linear space of $\nabla R$ at a point is the direct sum of representations with 
highest weights $(1), (2,1), (3), (3,2).$ So in some ways Riemannian and also conformal geometry
is closely linked to the representation theory of the orthogonal group.

Now for the conformal geometry we need to see how these geometric objects 
change under conformal change of the metric tensor. Already in dimension two (where the
scalar curvature is twice the Gauss curvature) it is not completely obvious; of course it
has to be such that the integral of the scalar curvature over a compact Riemann 
surface does not change, since it is a constant times the Euler characeristic. And indeed, we have
$$\Delta \omega + J = J_{\omega}e^{2\omega}$$
for the Gauss curvatures $J$ and $J_{\omega}$ of $g$ and $\overline{g} = e^{2 \omega} g.$
Hence integration over $M$ with respect to the natural volume $dvol_g$ from the metric will
become $\int_M J dvol_g = \int_M J_{\omega} dvol_{\overline{g}}$ as desired. Note that
$\int_M \Delta \omega dvol_g = 0.$ This is our
first example (although somewhat trivial) of a functional $F(g)$ which is conformally invariant.
Our aim is to find more of a similar nature.

\begin{lemma} The Yamabe operator on $M$ of dimension $n$
in the metric $g$ with Laplace operator $\Delta$
and scalar curvature $K$
$$Y = Y(g) = \Delta + \frac{n-2}{4(n-1)} K$$
satisfies
$$Y(\overline{g}) = Y(e^{2\omega}g) = e^{-b\omega} Y(g) e^{a\omega}$$
with  $a = (n-2)/2, \, b = (n+2)/2$.
Hence the Yamabe operator $Y$ is conformally covariant with bidegrees $a = (n-2)/2, \, b = (n+2)/2$.
\end{lemma}
Here the basic identity is for the change of the Levi-Civita connection, namely
$$\overline{\nabla}_X Y  = \nabla_X Y + d\omega (X) Y + d\omega(Y)X - g(X,Y) \nabla \omega$$
with the gradient defined as the dual of the derivative as $d\omega (X) = g(\nabla \omega, X)$ for
all vectors $X.$
Recall the well-known formula in two dimensions for the conformal change of the Laplace operator:
$\Delta(e^{2\omega}g) = e^{-2\omega} \Delta (g)$, i.e. just a special case of the Yamabe relation.
A subtle point, but as we shall see significant, is that differentiating the Yamabe relation
with respect to the dimension $n$ at $n = 2$, and applying both sides to the
constant function $1$, leads to the formula for the conformal change of
Gauss curvature, which is a remarkable formula, since it is linear in the function $\omega.$

\begin{remark} The study of scalar curvature is a classical subject with a large literature;
we mention here two important aspects, related to the above, namely
\begin{itemize} 
\item Can we on a compact Riemannian manifold $(M, g)$ find a conformally related metric
with constant scalar curvature? 
\item Which functions $K$ can be the scalar curvature for some metric on a manifold $M$, say
on the $n$-sphere $S^n$?
\end{itemize}
\end{remark}
The first is the celebrated Yamabe conjecture, solved in the affirmative by the efforts of many workers,
in particular by R. Schoen (see \cite{bra_1} and references therein, also as a general reference to
conformal geometry). Note that it amounts to solving, due to Lemma 1, applied to the constant function
$1$, for a positive smooth function $\phi$, and $c_n = (n-2)/4(n-1)$
$$(\Delta + c_n K)\phi = c_n \overline{K} \phi^{(n+2)/(n-2)}$$
namely with $\phi = e^{\omega (n-2)/2}$ and $\overline{K}$ the scalar curvature
of the conformally deformed metric (to be a constant). Here $n>2$, the statement in
two dimensions being the known uniformization of a surface.

The second question is also quite deep, see \cite{chang}, where the already difficult case of the
standard $2$-sphere is treated. Note also the necessary condition, related to the
Pohozaev identity, generalized in \cite{gover_ors}. This is the condition $\int_M XK dvol = 0$
for the scalar curvature $K$ and $X$ a conformal vector field, i.e. generating a (local) conformal
transformation. We shall see this in conncection with our study of the heat equation below.

\subsection{Asymptotics of heat kernels and their conformal variation}
Central references here are \cite{ro_pa} and \cite{bran_ors3}. See also the excellent survey paper
\cite{bra_2}. For the basic theory of
elliptic operators on manifolds see \cite{gil}. Consider on a compact Riemannian manifold
$(M, g)$ an elliptic partial differential operator $D$ of order $d$ with a positive definite leading symbol
$p_L$ and total symbol $p$; it could be
on functions, i.e. in the space $L^2(M)$ with respect to the Riemannian volume, or in
$L^2$ - sections of a Hermitian vector bundle $V$ over $M$. We assume $D$ is self-adjoint.
This then has a discrete set of eigenvalues
$\lambda_1 \leq \lambda_2 \leq \dots$ (repeated if there are multiplicities) tending to 
infinity and corresponding eigenfunctions/sections. In fact, we have the Weyl
asymptotic law $\lambda_j \sim C j^{d/n}$ for $j \to\infty$ for some constant $C$. Now
the asymptotics
of these eigenvalues have a profound influence (via the elliptic theory, and also
via Tauberian theory) on the small lime behaviour of the heat equation
$$(\partial_t + D)h(x,t) = 0, \, lim_{t \to 0}h(x,t) = f(x)$$
for $t>0$ and some smooth initial condition $f(x).$
The solution is given by a smooth kernel
$$h(x,t) = \int_M K(t,x,y) f(y) dvol(y)$$
with an asymptotic expansion, as in Lemma 1.8.2 in \cite{gil} as follows
\begin{prop} For a $D$ as above we have for the heat kernel and small positive $t$
$$K(t,x,x) \sim \sum_{k = 0}^{\infty}
 e_k(x) t^{(k-n)/d}$$
in the asymptotic sense, that the error can be made to any power of $t$ by summing enough
terms. Here the important local invariants $e_k(x)$ are given in terms of the total symbol $p$ of $D$
as invariant expressions in the jets of $p$. These invariants vanish for odd $k$.
\end{prop}
Now we may study how this expansion depends on the metric within a conformal class, namely
by letting $D$ be conformally covariant. Consider a conformal curve of metrics $g(u) = e^{2u\omega}$
with a real parameter $u \in \R$ and a fixed smooth function $\omega \in C^{\infty}(M, \R)$. Now
all quantities depend on $u$, such as connection, curvatures, diffferential operators, 
and corresponding invariants, and the dependence will be smooth in the relevant sense. For a quantity ${\cal C} = {\cal C} (u)$ we shall
denote its derivative at zero by ${\cal C}^{\bullet} = \frac{d}{du} {\cal C} (u)|_{u = 0}.$
The covariance of $D$ means that $D^{\bullet} = - b \omega D + a D \omega$ with $b - a = d$,
the order of $D$. 
Formally we have
\begin{equation*}
\begin{split}(Tr_ {L^2} e^{-tD})^{\bullet} &= \sum _{i = 0}^{\infty} t^{(2i - n)/d}
(\int_M U_i dvol_g )^{\bullet}\\ 
                                                           &= Tr_{L^2} -t D^{\bullet} e^{-tD}\\
                                                           &= Tr_{L^2} -t (-b \omega D + aD\omega)e^{-tD}\\
                                                           & = Tr_{L^2}t(b-a) \omega D e^{-tD}\\
                                                           & = -t\frac{d}{dt} Tr_{L^2} (b-a) \omega e^{-tD}\\
                                                           & = d\sum_{i = 0}^{\infty} \frac{n- 2i}{d} 
t^{(2i - n)/d}\int_M \omega U_i dvol_g.
\end{split}
\end{equation*}
We conclude by equating coefficients to equal powers of $t$ that
$$(\int_M U_i dvol_g)^{\bullet} = (n-2i)\int_M \omega U_i dvol_g.$$
Note that  in even dimension $n$ we get that $\int_M U_{n/2} dvol_g$ is a conformal
invariant - and exactly at the {\it index level}, i.e. the coefficient to the power $t^0$ - we may
call this the {\it conformal index}.

There are some important technical issues in these formal calculations; the equalities should be read as
between asymptotic expansions, and the term by term derivatives can be justified by going back to the
actual construction of the expansion of the heat kernel from the elliptic theory and corresponding Sobolev spaces on $M$; one has to make all estimates locally uniform in $u \in \R$, and also take into account the asymptotic nature in $t \to 0$.
For this see Theorem 3.3 in \cite{bran_ors3}.

Based on the similar analytic facts about elliptic operators we may also consider the spectral zeta function
$$\zeta_D(s) = \sum_{j = 1}^{\infty} \lambda_j^{-s}$$
of a complex variable $s \in \C$, assuming for simplicity that the spectrum of $D$ is strictly
positive, $0 < \lambda_1 \leq \lambda_2 \dots$ where the asymtotics will make this sum convergent
in some right half-plane in $\C$. Indeed, it will be (like the Riemann zeta function) meromorphic
in $s \in \C$ and regular at $s = 0$; this is seen from the heat asymptotics and using a Mellin
transform of this to express the zeta function. In particular we now define the {\it zeta regularized
determinant} of $D$ as
$$det(D) = e^{-\zeta_D'(0)}$$
i.e. in terms of the $s$-derivative at $s = 0$. This is an important quantity in the physics of
strings, and more. Note that for the Riemann zeta function this would say that
$$1\cdot 2 \cdot 3 \cdot \dots = \sqrt{\pi}$$
with this formal understanding of the infinite product.

Now we may in a similar way as for the heat invariants consider the conformal variation
of the quantity $\zeta_D'(0)$ as it depends on the metric during a conformal change.
Using
$$\zeta_D(s) = \frac{1}{\Gamma(s)} \int_0^{\infty} t^{s-1}Tr e^{-tD} dt$$
where we have $Tr e^{-tD} = \sum_{j = 1}^{\infty}e^{-t \lambda_j}$ and also
expressed in terms of the previous heat kernel, we get (write the the order of $D$ as $d = 2l$) 
$$\zeta_D'(0)^{\bullet} = 2l\int_M \omega U_{n/2} dvol_g.$$
Note that for the dimension $n$ odd, this is zero since the heat invariants
vanish at this level, so we get that under the assumptions (of covariant $D$) that
the determinant is a conformal invariant. Note also that with the variation in even dimension 
of the determinant we in effect now have a differential equation for $\zeta_D'(0)$, if we can
write the expression $U_{n/2}$ in an explicit and effective way along a conformal class of metrics. Note finally that  $\zeta_D(0) = \int_M U_{n/2} dvol_g$, i.e. the conformal index.
We summarize in the following (see \cite{bran_ors1} and \cite{bran_ors3})
\begin{thm}
Let $D$ be a positive elliptic differential operator on a compact Riemannian manfold $(M,g)$ of 
(necessarily) even order $d =2l$, and
assume it is built in a universal way from the metric, and that it is conformally covariant (or a positive
integer power
of a conformally covariant operator). Then in its asymptotic heat expansion
$$Tr_ {L^2} e^{-tD}) \sim
 \sum _{i = 0}^{\infty} t^{(2i - n)/d}
(\int_M U_i dvol_g )$$
the coefficient functions ({\it heat invariants}) $U_i$ are invariant polynomials in $R, \nabla R, \dots $,
i.e in the curvature and its covariant deirvatives; along a conformal curve of metrics $e^{2u\omega}g$
the derivative at $u = 0$ is given by
$$(\int_M U_i dvol_g)^{\bullet} = (n-2i)\int_M \omega U_i dvol_g.$$
For the variation of the zeta function (negative log - determinant of $D$) we have
$$\zeta_D'(0)^{\bullet} = 2l\int_M \omega U_{n/2} dvol_g.$$
\end{thm}
\begin{remark} These heat invariants are not the only source of conformally invariant
functionals of the form $F(g) = \int_M I (x) dvol_g(x)$; they appear also in  the
celebrated Branson $Q$-curvatures, see e.g. \cite{bra_2}. The key object is here
a conformally covariant differential operator on functions of the form $P_n = \Delta^{n/2}$ + lower order
in even dimension $n$, these are fundamental objects in conformal geometry, called GJMS operators,
see \cite{bra_2} p. 18 for a discussion and references.
These remarkable operators were first found in \cite{gjms} and later explained more in the likewise
fundamental papers \cite{fg1} and \cite{fg2}.
The lower order terms are curvature corrections,
so in flat space it is just a power of the Laplacian (also in flat space with a metric with signature). More
generally, there are GJMS operators of the form 
$P_{2k} = \Delta^k$ + lower order, see \cite{bra_2} and references therein. Here any $k$ is
allowed in odd dimension, and in even dimension $n$ only $k \leq n/2$; there is the so-called
Fefferman-Graham obstruction tensor ${\mathcal O}_{ij}$, a symmetric, trace-free divergence-free
two-tensor, which has to vanish (as it does on Einstein manifolds) in order to ensure
existence in even dimension for higher values of $k$. There is a nice relation to
the Branson $Q$-curvature (see below), namely for a metric variation of the metric $g_{ij} = g_{ij}(t),
t \in \R$ with the derivative $h_{ij}$ at $t=0$ we have
$$\frac{d}{dt}_{t=0}\int_M Q dvol = (-1)^{n/2} \frac{n-2}{2}\int_M {\mathcal O}_{ij}h^{ij} dvol.$$
There are
at least two proofs of the GJMS existence whereas actual expressions are very complicated
For further backgrond on the GJMS operators and the $Q$ curvature see \cite{juhl} and 
\cite{juhl2}.
See also the very recent survey
\cite{case_gover} for an expert treatment of both the background of these operators, the
applications and recent progress, and some interesting future perspectives.
These are in particular also intertwining operators between bundles with the appropriate conformal weights.
See the discussion in the Section on the minimal representation of $O(p,q).$
The important relation is (as for the Gauss curvature in two dimensions, $n = 2$)
$$P_n \omega +Q_n(g) =Q_n(e^{2 \omega}) e^{n\omega}.$$
Since $P_n$ is self-adjoint and annihilates constants it follows that the integrated $\int_M Q_n(x) dvol_g(x)$
is conformally invariant. As one might suspect, such local invariants $I(x)$ are very special
expressions, and in fact only involve summing the Pfaffian (whose integral is the Euler characteristic),
pure divergences, and quantities that scale locally in a way opposite the volume under conformal
changes. See \cite{alex} for a tour de force on this topic.
\end{remark}
Let us point out a corollary to the above variational results, namely
a conservation law for conformal vector fields, see \cite{bran_ors1} Theorem 5.4, and also
\cite{gover_ors}. Let $X$ be a conformal vector field on $M$, i.e. it generates a local
group of conformal transformations (sometimes called a conformal Killing vector field), then
it preserves the integrals of the heat invariants as follows
\begin{thm} Let $D$ be as in the previous Theorem and $U_i$ the
corresponding heat invariants, then with $X$ a conformal vector field and for all $i$
$$\int_M X \cdot U_i dvol_g = 0.$$
\end{thm}

The main example is here $D = Y$, the Yamabe operator with $d = 2$, or it could also be the
square of the Dirac operator, a prime example of a first-order covariant operator 
(acting on spinor fields). For these some details of the heat coefficients are in \cite{bran_ors2}.
If our operator is the Yamabe operator (or generally of the form $\Delta + aK$), then the heat
invariants are  $U_0 = (4 \pi)^{-n/2}, U_1 = (4 \pi)^{-n/2}(\frac{1}{6} - a)K,
U_2 = (4 \pi)^{-n/2}\frac{1}{180}(90 (\frac{1}{6} - a)^2 K^2 - |r|^2 + |R|^2 
- 30(\frac{1}{5} - a)\Delta K).$ These expressions for $U_i$ at higher levels $i$ very
quickly become complicated; noneless we can use them, especially in low dimensions,
to study corresponding aspects of conformal geometry.

One such is the question of extremal properties of the determinant, in particular for the
Yamabe operator on standard spheres. This can be done rather explicitly for the standard
spheres in low dimension, beginning with the formula of Polyakov, see \cite{bran_ors2},
$$- {\rm log}\frac{det \Delta(e^{2 \omega}g)}{det \Delta(g)} = \frac{1}{12 \pi}
\int_M (|\nabla \omega|^2 + 2K \omega) dvol_g$$
for a two-dimensional closed surface $M$ with Gauss curvature $K$.
This is obtained from the variational formula in Theorem 4 by integration along
a conformal curve of metrics. Let us pause to see how this works concretely.

Fix a metric $g$ on the Riemann surface $M$ (without boundary - actually there is
also an extension of the results to the case of $M$ with boundary), let $\Delta$ be the Laplace 
operator and $K$ the Gauss curvature. Then for the zeta function $\zeta(s) = \zeta_A(s)$ corresponding to
{\it any} conformally covariant elliptic differential operator $A = A(g)$ (or an integer power of such) -
it could be $\Delta$ itself - we consider its dependence on $u \in \R$ as we make the conformal change
of metric $g(u) = e^{2u\omega}g$ for $\omega$ a fixed smooth real function on $M.$
Since we know that the only heat invariant at  this level is the Gauss curvature, we have 
$\zeta^{\prime} (0)^{\bullet} = c \int_M \omega K dvol_g$
(for some constant depending on $A$)
and we can also take the derivative at any $u \in \R$, viewing the conformal factor as 
$e^{2(u+\epsilon)\omega}$ and taking the $\epsilon$-derivative at $\epsilon = 0.$
This leads to
$$\frac{d}{du} \zeta^{\prime}(0) = c \int_M \omega K_{u \omega}e^{2u\omega}dvol_g,$$
where $K_{u \omega}$ is the Gauss curvature of the metric $g(u) = e^{2u\omega}g.$
Recall the formula for the Gauss curvature
$$\Delta u \omega + K = K_{u \omega} e^{2u \omega}$$
which gives us
$$\frac{d}{du} \zeta^{\prime}(0) = c \int_M u \omega \Delta \omega + K \omega dvol_g.$$
We integrate in $u$ from $u = 0$ to $1$ and obtain
$$\zeta_{\overline{A}}^{\prime}(0) - \zeta_{A}^{\prime}(0) = 
c \int_M \frac{1}{2} \omega \Delta \omega + K \omega dvol_g.$$
Here $\overline{A} = A(e^{2\omega}g).$ This is the Polyakov formula, with the constant $c$ 
for the Laplace operator. It is remarkable that for the square of the Dirac operator the
constant has the opposite sign, and the extremal property is the opposite as for the
Laplace operator.

Now it is fortunate that on the standard  two-sphere $S^2$
we may estimate this quantity via an inequality due to Moser and Trudinger to the effect that:
For any metric conformal to the standard one on $S^2$ with the same volume, the
determinant of the Laplacian is maximized exactly at the standard metric itself, or
its pull-back by a conformal map on the sphere. More generally one may think of this 
procedure of maximizing the determinant in a conformal class as explaining the
uniformization of Riemann surfaces, see \cite{osgood}.

Now there are similar results in 4 and 6 dimensions, this time using the Yamabe operator and
the explicit heat invariants. In particular it holds that the determinant of the Yamabe operator
is minimized at the standard metric (and its conformal transforms) among volume-preserving
conformal changes on $S^4$ and maximized on $S^6$ - and it is expected to be minimized on
$S^8$ and so on. See \cite{bra_2} for much more on this, and references therein. In particular
we quote \cite{bran_chang} for the Yamabe operator $Y$ and the square $D^2$ of the Dirac operator.
The fact that one can find explicit expressions of the determinants in four and six
dimensions, relies on the transformation
property of the Branson $Q$-curvatures, similar to that for Gauss curvature; in particular notice
the linear term in $\omega.$ $Q$ is an important term in the heat invariants. For more on this see
\cite{bran_ors2} where the role of the heat invariants and the explicit expressions for
determinants in four dimensions are given, involving integrals of $Q$.

\begin{thm} On $S^4$, the standard metric minimizes 
$det Y$ and maximizes
$det D^2$ among all conformal metrics of a fixed volume.
\end{thm} 

\begin{remark} There are some interesting inequalities behind these statements, now from
Fourier analysis in $\R^n$ in general, and some applications to four dimensions,
in particular on the sphere,
using conformal geometry. See the important paper \cite{beck} Section 3, and \cite{bran_ors2}. Consider the following quantities
$$S_1(F) = \frac{1}{48} \int_{S^4} |\Delta F|^2 d \xi + \frac{1}{24}\int_{S^4} |\nabla F|^2 d \xi
+ \int_{S^4} F d \xi - {\rm log} \int_{S^4} e^F d \xi$$
and
$$S_2(F) = \int_ {S^4} e^{-F/2} |\Delta e^{F/4}|^2 d \xi - \frac{1}{4}\int_ {S^4}|\nabla F|^2 d \xi$$
for a real smooth function $F$ on the four-sphere, and the integration is with respect to normalized surface measure.
Now these are both non-negative, the first (an analogue of the Moser-Trudinger inequality) by
an ingenious argument based on the Hardy-Littlewood-Sobolev inequality, 
and the second by the sharp Sobolev inequality.
We may explain a little of the background for these inequalities, since they also pertain to
representation theory, namely the properties of the convolution operators
$$I_{\lambda}: f(x) \mapsto \int_{\R^n}\frac{f(y)}{|x-y|^{\lambda}}dy,$$
called Riesz operators or Knapp-Stein operators. 
The key estimate is due to E. H. Lieb,
see \cite{beck} and references therein and further discussion, namely
$$|\int_{\R^n \times \R^n}\frac{f(x)g(y)}{|x-y|^{\lambda}} dx dy| \leq A_p ||f||_{L^p(\R^n)}
||g||_{L^p(\R^n)}$$
$$A_p =  \pi^{n/p'}\frac{\Gamma(n(\frac{1}{p} - \frac{1}{2}))}{\Gamma(\frac{n}{p})}
\bigg(\frac{\Gamma(\frac{n}{2})}{\Gamma(n)}\bigg)^{1 - (2/p)}$$
where $\lambda = 2n/p', 1 \leq p < 2, 1/p + 1/p' = 1.$ This HLS inequality is sharp with
extremal functions $f(x) = g(x) = (1+|x|^2) ^{-n/p}$ up to the
action of the conformal group on $\R^n$, which is $SO(n+1,1).$ Now this
inequality should be thought of in the following way, namely we 
have two norms both invariant under an isometric
action of $SO(n+1,1)$, one on the left-hand side is the Hilbert norm for the complementary series, and the
other the $L^p$ norm. The inequality may also be expressed on $S^n$ via
stereographic projection (a conformal map), and furthermore there are limit
inequalities at each endpoint, namely for $p = 2$
$$\int_{S^n}|F(x)|^2 {\rm log}|F(x)| dx \leq \sum_{k= 1}^{\infty} \Delta_k(n) \int_{S^n} |Y_k(x)|^2 dx$$
where $F = \Sigma_{k=0}^{\infty} Y_k$ with $Y_k$ spherical harmonic of degree $k$,
integration with respect to normalized surface measure, and $\int_{S^n} |F(x)|^2 dx = 1$, and
$$\Delta_k(n) = \frac{n}{2} \sum_{l = 0}^{k-1} \frac{1}{\frac{n}{2} + l}.$$
At the other end $p = 1$ one obtains
$${\rm log}\int_{S^n} e^{F(x)} dx  \leq \int_{S^n} F(x) dx + \frac{1}{2n} 
\sum_{k=1}^{\infty} \frac{\Gamma(n+k)}{\Gamma(n)\Gamma(k)}\int_{S^n} |Y_k(x)|^2 dx$$
which corresponds to the quantity $S_1$.
Note that when $n$ is even this can also be written in terms of the GJMS operator $P_n$ as
$${\rm log}\int_{S^n} e^{F(x)} dx  \leq \int_{S^n} F(x) dx + \frac{1}{2n!} 
\int_{S^n}F(P_nF) dx.$$ 
The other $S_2$ will be estimated by a sharp 
Sobolev inequality, another consequence of the HLS. The first inequality from $p = 2$ is
called a logarithmic Sobolev inequality, and indeed leads to the celebrated logarithmic Sobolev
inequality of L. Gross for Gauss measure. Let us finally remark that if we consider the intertwining
operator $I_{\lambda}$ on the sphere $S^n$ (with the same expression as in $\R^n$ now with
$|x-y|$ the Euclidian distance) and normalize it to have $I_{\lambda} 1 = 1$ then the HLS says
that, writing $q = p'$ for the dual exponent, $$||I_{\lambda} F||_q^2 \leq (I_{\lambda} F, F) \leq ||F||_p^2$$
relative to the norms in $L^q(S^n)$ resp. $L^p(S^n)$ and the $L^2(S^n)$ pairing; and
the equality is exactly for $F = 1$ or its orbit under $SO(n+1,1).$ In particular the Knapp-Stein
intertwining operator is a contraction from $L^p$ to its dual $L^q.$
We can be more explicit here, since the eigenvalues for (the normalized) $I_{\lambda}$ 
are known on the spherical harmonics $Y_k$, via the Funk-Hecke formula, or the spectrum 
generating argument in \cite{bran_ola}, namely
$$\gamma_k = \frac{\Gamma(\frac{n}{p})\Gamma(\frac{n}{q} + k)}
{\Gamma(\frac{n}{q})\Gamma(\frac{n}{p} + k)}.$$
So with $F$ as above the HLS is saying that
$$\sum_{k = 0}^{\infty} \gamma_k \int_{S^n}|Y_k(x)|^2 dx \leq ||F||_p^2$$
with equality if and only if $F = 1$ or in its orbit under $SO(n+1,1).$
Let us finally see how the HLS can be expressed in a way to make the GJMS operators
on the sphere appear as residues of the meromorphic intertwining operator family $I_{\lambda}.$
We express the dual inequality as
$$||F||_q^2 \leq \int_{S^n} F (D_{2r}F) dx$$
where $r = n/2 - n/q$ so that the operator $D_{2r}$ has eigenvalues on spherical harmonics of degree $k$
$$\gamma_k^{-1} = \frac{\Gamma(\frac{n}{2} - r)}{\Gamma(\frac{n}{2} + r)}
\frac{\Gamma(k +\frac{n}{2} + r)}{\Gamma(k + \frac{n}{2} - r)}.$$
Here for $r \in \N, r < n/2$ the $D_{2r}$ is a differential operator, up to a constant the GJMS
operator on $S^n$ of order $2r.$

It is a strong result that for the
group $G = SU(n+1,1)$ the analogous estimates hold, see \cite{frank_lieb} for some careful
analysis on the Heisenberg group. 

Now the geometric content of $S_1$ and $S_2$ is that if we let $g$ be the
standard metric and $\overline{g} = e^{F/2} g$ conformally related, then the ratio of
determinants for the Yamabe operators $Y$ resp. $\overline Y$ is governed by these functionals as 
$$det \overline Y / det Y = e^{-\beta_1 S_1(F) - \beta_2 S_2(F)}$$
with $\beta_1 = \beta_2 = -180.$ If we replace the Yamabe operator $Y$ by the square of the Dirac
and find the same ratio, then the values are $7\beta_1 = 22 \beta_2 = 77/180$, and hence the statements
about the minimum resp. maximum. This pattern is expected to continue.

\end{remark}

\subsection{Determinants and their rigidity}In this section we shall see as in \cite{moe} how the conformal
differential geometry and representation theory of the conformal group of the standard $n$-sphere
$S^n$ may be combined to obtain local extremal properties for determinants and other
conformal functionals. For background on determinants see \cite{bran_ors2}. This time we shall
consider arbitrary variations of the metric $g$, and not just in directions of conformal change.
We consider functionals
$$F : {\rm Metr(M)} \mapsto \R$$
on the space of smooth Riemannian metrics on $M$, a compact manifold, equipped with a smooth
topology as in \cite{moe} from Sobolev spaces on the space of sections of the bundle of
symmetric two-tensors $C^{\infty}(S^2 TM)$. The tangent space at a metric $g$ is
naturally identified with $C^{\infty}(S^2 TM)$ and has an inner product
$ <h,k>_g = \int_M (h,k)_g dvol_g$ via the pointwise inner product on tensors.
We now make the {\bf assumptions} (i) $F(\phi^* g) = F(g)$ for all $g$ and all $\phi \in {\rm Diff}_0(M)$,
the identity connected component of the group of diffeomorphisms of $M$, and  (ii)
$F(e^{2\omega}g) = F(g)$ for all $g$ and conformal factors $e^{2\omega}$.
This holds for the functionals considered earlier, namely just the independence of coordinates, and the conformal
invariance. Examples could be a conformal index for a covariant operator,
or the determinant in odd dimension of such an operator.
For such a functional we want to find the gradient (first derivative) and the Hessian
(second derivative); these are obtained by differentiating at a base point $g \in {\rm Metr}(M)$
with respect to $t$ the function
$$F(g + tk) = F(\phi^* g + t\phi^*k) = F(\Omega_{\phi}^2 g + t \phi^*k)
= F(g + t\Omega_{\phi}^{-2}\phi^*k)$$
where $\phi \in G = {\rm Conf}_0(M)$, the group of conformal transformations of $(M,g)$
with $\phi^*g = \Omega_{\phi}^2 g$, and
$k$ a tangent vector at $g$. Note that here appears an action of $G$ on the tangent vectors $k$ -
and that these are sections of a vector bundle over $M$. 
If $\phi^*g = \Omega_{\phi}^2 g$ we can define a representation of $G$ by
$$\pi_{\nu}(\phi^{-1})k = \Omega_{\phi}^{\rho + \nu -2}\phi^*k$$
for $\phi \in G, \, k \in C^{\infty}(S_0^2 TM)$ with $\rho = n/2, \, \nu \in \R.$

Now in the case of the sphere, i.e. 
$M = S^n$, this is a homogenous vector bundle and we may use representation theory
of $G$ to analyze the gradient and the Hession of $F$. 
The gradient is
$$DF_g(k) = \frac{d}{dt} F(g + tk)|_{t = 0}$$
and vanishes for all $k \in C^{\infty}(S^2 TM)$ at a stationary point. The next derivative gives the
Hessian $H_g$ via
$$D^2F_g(k,k) = \frac{d^2}{dt^2}F(g + tk)|_{t = 0} = <k,H_g k>_g$$
where we assume $C^3$ differentiability in the $t$-variable, see \cite{moe}. In the cases we shall do, the
base metric $g$ is stationary and $F$ sufficiently smooth, so that we get the symmetric
operator
$$H_g: C^{\infty}(S^2 TM) \mapsto C^{\infty}(S^2 TM).$$
 
Roughly speaking the tangent directions
from the base point $g$ are (i) directions of conformal change, (ii) directions of change
by diffeomorphisms, and (iii) the remaining directions, where by assumption the functional
is stationary in the first two. Our result will be that the functional (sufficiently smooth) will have a strict maximum
or minimum in a neighbourhood of $g$ in the remaining directions. Specifically we have the
directions tangent to conformal rescalings of $g$
$${\rm conf}_g = \{ \omega g| \omega \in C^{\infty}(M) \}\subset C^{\infty}(S^2TM)$$
and the orthogonal complement $({\rm conf}_g)^{\perp} = C^{\infty}(S^2_0 TM)$,
the trace-free symmetric two-tensors. Similarly, we have the directions tangent to
diffeomorphism pullbacks of $g$
$${\rm diff}_g = \{L_X g | X \in C^{\infty}(TM)\}$$
where $L_X$ denotes the Lie derivative along the vector field $X$. Note that the vector
field $X$ is conformal means $L_X = 2\omega_X g$ with the infinitesimal cocycle $\omega_X$
a smooth function, equal to the divergence of $X$ up to a factor $n$. 
Now the action of the Lie algebra of $G$,
identified with the conformal vector fields ${\rm cvf}(M, g)$ is
$$d\pi_{\nu}(X) k = L_X k + (\rho + \nu -2)\omega_X k$$
namely the derivative of the earlier $\pi_{\nu}$. Note that in general
${\rm cvf}(M, g)$ and $G$ are finite-dimensional, so a Lie algebra resp. a Lie group.
For the sphere $S^n$ the conformal group is $ G = SO(n+1,1)$ via the action by linear fractional
transformations, as we shall see explicitly below.  

We also introduce in general the {\it Ahlfors operator}
$$S_g : C^{\infty}(TM) \mapsto C^{\infty}(S^2_0 TM)$$
by $S_ g X  = L_X g - \frac{2}{n} (div_g X) g$ with $ker S = {\rm cvf} (M,g)$.
From \cite{moe} we have 
\begin{lemma} $S = S_g$ is a first-order intertwining differential operator
$$\Omega_{\phi}^{-2} \phi^* SX = S \phi^* X$$
for $\phi^* g  = \Omega _{\phi}^2 g$, and the image $ran S \subset ker H_ g$
under our assumptions on the corresponding functional $F$. Furthermore,
the dual to $S$ is $S^*k = 2 div k$. 
\begin{remark} These two operators $S$ and $S^*$ are examples of very general
first-order intertwining operators over flag manifolds,
obtained as composition of the covariant derivative and
a projection in a tensor product, see \cite{ors} and references therein; they are
sometimes called generalized (or Stein-Weiss) gradients. Recall the Levi-Civita connection
in Riemannian geometry
$$\nabla: C^{\infty}(E) \mapsto C^{\infty}(T^* \otimes E)$$
and let us consider the bundle $E$ to be the bundle $T = TM$, i.e. the tangent bundle. The
tensor product in the fiber is $\R^n \otimes \R^n$ i.e. highest weight $(1) \otimes (1)$
which decomposes under $SO(n)$ as $(0) \oplus (2) \oplus (1,1)$, namely the trivial, the trace-free
symmetric two-tensors, and the skew-symmetric two-tensors (two-forms). We can now form
${\rm proj} \circ \nabla$ for each projection and obtain the divergence, the Ahlfors $S$ and the
de Rham derivative (identifying a vector field with a one-form). For each there is a unique conformal weight
$p$ making the operator conformally covariant. The operators in the de Rham complex arise in this way,
as does for example the Dirac operator.
\end{remark}
\end{lemma}
Now in general the symmetries are 
contained in the following main observation
\begin{lemma}
Let $F = F(g)$ be a sufficiently smooth functional on the space of Riemannian metrics on a fixed
compact manifold $M$ and assume $F$ is invariant under conformal change, and under
diffeomorphisms of $M$; then at a stationary point $g$ the Hessian satifies the intertwining relation
$$\pi_{n/2}(\phi) H_g = H_g \pi_{-n/2}(\phi)$$
for all $\phi \in G.$
\end{lemma}
\begin{proof}We see from the assumed 
invariance of $F$ that
$$<\Omega_{\phi}^{-2} \phi^* k, H_g \Omega_{\phi}^{-2} \phi^* k>_g = <k, H_g k>_g$$
and since $\phi^* dvol_g = \Omega_{\phi}^n dvol_g$ we get
$<h,k>_g = <\pi_ {-n/2}(\phi)h, \pi_{n/2}(\phi)k>_g$ from which
$\pi_{n/2} (\phi^{-1}) H_g \pi_{-n/2}(\phi) = H_g.$
\end{proof}

Now we specialize our compact manifold to be the sphere $S^n$ with the standard metric, and study using representation theory the
bundles in question as homogeneous vector bundles, i.e. viewed as induced representations
over the flag manifold $S^n$. The intertwining operators such as $H_g$ then become rather
explicit, in particular we can find 
their spectrum. Here $G = SO_0(n+1,1)$ with the action on $S^n \times \{1\}$ given by
$$h \cdot y = \frac{h(y,1)}{(h(y,1))_{n+2}} \in S^n \times \{1\}$$
for $h \in G$ and $y \in S^n \subset \R^{n+1}$. Then we get $S^n = G/P$ with
the parabolic subgroup $P = MAN, M \simeq SO(n), A \simeq \R^+, N \simeq \R ^n$
and an Iwasawa decomposition $G = KAN,  K \simeq SO(n+1)$. The roots
are $\pm \alpha$ and $\rho = n \alpha/2$. Characters of $A$ are of the form
$\chi_p (a) = a^p$ and then $\rho$ corresponds to $p = n/2$.

Let $(V_{\sigma}, \sigma)$ be an irreducible representation of $M$ and extend to $A$
by $(a,m) \to a^p \sigma(m), a \in \R^+, m \in SO(n)$,
and denote this representation by $(V_{\sigma}^p, \sigma^p)$. We call $p$
the conformal weight and define the associated vector bundles
$$\V _{\sigma}^p = G \times_P V_{\sigma}^p = {\rm Ind}_{MAN}^{KAN}(\sigma \otimes a^p \otimes 1)$$
$$C^{\infty}(\V_{\sigma}^p) = \{\psi \in C^{\infty}(G)| \psi(xman) = 
a^{-p-\rho} \sigma(m)^{-1} \psi(x)\}$$
with the usual shift in the induced picture. The action
of $G$ is  the left-regular action.
We can find the space $\mathcal{E} (S^n, \V_{\sigma}^p)$
of $K$-finite sections of the bundle by Frobenius reciprocity
$${\rm Hom}_ M(\sigma, \beta_{|M}) \simeq {\rm Hom}_K (\mathcal{E}(S^n, \V_{\sigma}^p), \beta)$$
where $\beta$ is  an irreducible representation of $K$. Here we have the usual (see
\cite{moe}) interlacing patterns
in the branching law from $K$ to $M$, and we are interested in the highest weights
$\sigma = (0) = (0, 0, \dots)$ (the trivial representation), $\sigma = (1) = (1, 0, 0, \dots)$
(the defining representation),
and $\sigma = (2) = (2, 0, 0, \dots)$ (trace-free symmetric two-tensors). The result is
for the $K$ types labeled by highest weights $\beta_1 \geq \beta_2\geq 0 \geq 0 \geq \dots$
\begin{prop}The $K$-types have multiplicity one in all bundles and given independently of $p$ by
$$\mathcal{E} (S^n, \V_{(1)}) = \oplus_{r = 0}^1 \oplus_{l = 0}^{\infty} E_ {1+l, r}$$
$$\mathcal{E} (S^n, \V_{(2)}) = \oplus_{q = 0}^2 \oplus_{j = 0}^{\infty} F_ {2+j, q}$$
where the subscripts in $E$ and $F$ are the highest weights in those representation spaces of
$K \simeq SO(n+1)$. 
Here we assume $n \geq 4$ (see \cite{moe} for the slight modifications when $n = 2, 3$).
\end{prop}
Now we observe the intertwining properties of the Ahlfors operator $S$ and the Hessian $H$
for our assumed functional in their action on the $K$-finite sections, forming a Harish-Chandra
module (the consistent action of the Lie algebra $\mathfrak g$ and $K$), here denoted
just by the $\V_{\sigma}^p$.
\begin{prop}We have as intertwining operators
$$S: \V_{(1)}^{-1} \mapsto \V_{(2)}^0$$ 
$$H: \V_{(2)}^0 \mapsto \V_{(2)}^n.$$            
In particular, $S$ and $H$ are diagonalized by the $K$-decompositions above, respectively.
Furthemore, provided that $H$ is non-zero, these two operators form an exact sequence, and
we have the explicit formula $H = c(F) T_0$ with a real constant $c(F)$, and
$T_0$ is the diagonal operator
$${T_0}_{|F_{(2+j, q)}} = \frac{\Gamma(n+j+2)\Gamma(n+q-1)}{\Gamma(j+2)\Gamma(q-1)} \geq 0$$
on the $K$ types $(2+j, q)$
\end{prop}
We call $T_0$ the {\it universal Hessian}, since it is independent on the functional $F$, and so the
nature of the stationary point (the base metric $g$) is governed by the proportionality constant $c(F)$.
\begin{proof} This is in \cite{moe}, which is based on the general theory of generating the spectrum of
intertwining operators for general homogeneous bundles over $S^n = G/P$, see \cite{bran_ola}. We already saw the
intertwining property of $S$ and $H$, here with the conformal weight $p$ made explicit. The
formula for the spectrum of $T_0$ should be seen as a meromorphic function, and note that
it is zero for  $q = 0, 1$, which is exactly the range of $S$.
\end{proof}

Finding the sign of the constants $c(F)$ for concrete functionals will determine whether we have local
maxima or minima of those functionals; at the same time the inverse function theorem will
give us strict maxima or minima in the directions perpendicular to the directions corresponding
to conformal change and diffeomorphism change, namely directions given by $\V_{(2)}^0/ran S$.
 Let $W$ be the orthogonal complement to
$ran S$, then as in \cite{moe} we have
 \begin{prop} Let $F$ be a sufficiently smooth conformal functional on $S^n$ with the standard metric
$g$ and $c(F) > 0$; then there is an open neighborhood $U \subset {\rm  Metr}(S^n)$ of $g$ such that
$F(g') \geq F(g)$ for all $g' \in U$, and $F(g') > F(g)$ for all $g \in W \cap U$.
\end{prop}

Finally we can determine the local extremal properties of some of the conformal functionals $F$
we have seen, namely heat invariants and their conformal indices, or the Branson
total $Q$-curvature, namely
$$F(g) = \int_M Q_g dvol_g$$
in even dimensions - this is a prominent case. In addition we have
determinants in odd dimension, here to be considered for conformally covariant operators on the sphere.
It turns out, that to find the sign of the constants $c(F)$ requires some delicate calculations
of the Hessians - essentially pseudodifferential operators of order $n$. Some earlier and important
work here was \cite{ok}. The point is that representation theory has simplified and
extended some of the arguments in \cite{ok}. The results are as in \cite{moe}, the first one already in
\cite{ok} for the Yamabe operator $Y$
\begin{thm} Among metrics on $S^n$ of fixed volume, the standard sphere $(S^{2k+1}, g)$
is a local maximum for $(-1)^{k+1} det Y$, and the standard sphere $(S^{2k}, g)$ is a local
maximum for $(-1)^{k +1}\zeta_Y(0)$.
\end{thm}
We have here the zeta function $\zeta_Y(s)$ for the Yamabe operator. Though we shall not go into the
definition of the Dirac operator $D$ (it is a generalized gradient, and conformally
covariant) let us record here that the same principles give, see \cite{moe} and \cite{moe2}
\begin{thm} Among metrics on $S^n$ of fixed volume, the standard sphere $(S^{2k+1}, g)$
is a local maximum for $(-1)^k det D^2$, and the standard sphere $(S^{2k}, g)$ is
a local maximum for $(-1)^k \zeta_{D^2}(0)$.
\end{thm}
Note the curious alternating behaviour in dimension and between determinants and
indices - and between Yamabe and Dirac squared. This should have some link to the
overall behaviour of the zeta functions (and similarly for the behaviour under conformal change
of determinants in even dimensions).
The condition of volume-preserving metrics corresponds to infinitesimal changes
of trace zero. Of course we still have the strict inequalities in directions orthogonal to
conformal or diffeomorphism change. To conclude this discussion of conformal functionals,
let us mention the Branson total $Q$-curvature $F(g) = \int_M Q_g dvol_g$ in even dimensions.
This too has the universal Hessian, which leads to a local maximum at the round sphere, \cite{moe}.

All of the above is only the beginning of much conformal geometry. Indeed, there are extensions of the theory
to conformal geometry of hypersurfaces which also generalize the classical concept of
Willmore energy of surfaces in Euclidian space; see \cite{blitz} and \cite{juhl_ors} for analogues of
GJMS operators and $Q$-curvature.

\section{Intertwining operators and minimal representations} From the beginning representation
theory of Lie groups was connected to differential equations via their symmetries. The special
relativity gave models of particles in terms of representations of the symmetry groups, such
as the Lorentz group $SL(2, \C)$ or the Poincare group $P = SL(2, \C) \times_s \R^4$
(semi-direct product with Minkowski space $\R^4 \simeq \R^{1, 3}$, $SL(2, \C)$ double-covering
$SO_0(1,3)$). Particles and the corresponding irreducible unitary representations of $P$ were labeled
by their spin $j = 0, \frac{1}{2}, 1, \frac{3}{2}, \dots$ and mass $m \geq 0$. It was soon realized
that the particles of mass zero had some additional symmetries, namely that the representations
could be extended to the conformal group of Minkowski space, locally an action of
the indefinite orthogonal group $O(2, 4)$ (strictly speaking the connected component
$SO_0(2,4)$ of the identity).
These representations were
realized in Hilbert spaces of solutions to hyperbolic partial differential
equations, admitting the symmetries from special relativity, such as the wave equation, as we shall see below. 
Also Maxwell's equations for the electro-magnetic field 
were conformally invariant, as were the spin zero and mass zero equations. This corresponds well
with the fact that the classical theory of light is modeling light rays as null geodesics with respect to the
Lorentz metric $ds^2 = dt^2 - dx_1^2 - dx_2^2 - dx_3^2.$ This means that light propagates
along a geodesic whose tangent is annihilated by the metric. One checks that this notion is
conformally invariant, i.e. multiplying the metric (in fact any Lorentz metric
of signature $(+, -, -, \dots ,-)$ in any dimension)
by a smooth positive function will lead to the same (reparametrized) null geodesics.

Now the wave equation is
$$\square \varphi = (\frac{\partial^2}{\partial t^2} - \frac{\partial ^2}{\partial x_1^2}
 - \frac{\partial ^2}{\partial x_2^2} - \frac{\partial ^2}{\partial x_3^2})\varphi = 0$$
with $\varphi = \varphi(t, x_1, x_2, x_3)$. $\square$ is here the Yamabe operator $Y$ in
flat Minkowski space. There is a Hilbert space of (distribution) solutions to the wave equation
carrying an irreducible unitary representation of $SO_0(2,4)$, which on $P$ is
the spin zero, mass zero relativistic particle. This is the minimal representation
of the conformal group that we shall study below for $O(p,q)$. For details see
\cite{Koba1}, \cite{Koba2}, \cite{Koba3}, where the point of view is to give as many
perspectives on this particular representation as possible; each model has its virtues,
in particular for the branching laws to symmetric (and other) subgroups.

\subsection{Constructions of the minimal $O(p,q)$ representation} The wave equation in
special relativity is the Yamabe operator in Minkowski space; now the idea is to extend this to
the Yamabe operator and its kernel, this time on $S^{p-1} \times S^{q-1}$. This model
was introduced in \cite{bin}, and it has the advantage to give the $K$-types in the representation 
from the beginning, $K = O(p) \times O(q)$, unlike the wave equation in flat space, as
explained above. Thus the manifold is $M = S^{p-1} \times S^{q-1}$ with the indefinite
pseudo-Riemannian product metric 
(the standard metric on the first sphere and the negative standard
metric on the second), and $Y_M$ its Yamabe operator, sometimes called an ultra-hyperbolic operator
(hyperbolic for $p = 2$). The entire discussion of natural differential operators carries over,
including connections and curvatures, to the indefinite case, and again the Yamabe operator
is conformally covariant; in particular for $G$ the group of conformal transformations of any
pseudo-Riemannian manifold $(M, g)$ with a left action $x \to hx = L_h x, x \in M, h \in G$,
$L_h^*g = \Omega(h, \cdot)^2 g$, we define a representation of $G$ by
$$\pi_{\lambda}(h^{-1})f(x) = \Omega(h, x)^{\lambda} f(L_h x)$$
say for $f \in C^{\infty}(M)$. As before, the Yamabe operator $Y = Y_M$ will be an
intertwining operator, i.e.
$$Y \pi_{\frac{n-2}{2}}(h) = \pi_{\frac{n+2}{2}}(h) Y, \, h \in G$$
and the kernel of $Y$ is a subspace invariant under $\pi_{\frac{n-2}{2}}(h), h\in G$.
Now it is a general principle that a (pseudo-)Riemannian manifold $M$ of dimension $n$
rarely has a large group ${\rm Conf}(M)$ of conformal transformations, always a Lie group, and this has dimension at most $(n+2)(n+1)/2$. But our product of spheres is such a case.
Indeed, ${\rm Conf}(S^{p-1} \times S^{q-1}) = O(p,q)$, which is seen in the following way.

Consider $\R^{p,q}$ with the metric
$$ds^2 = dx_1^2 + \dots + dx_p^2 - dy_1^2 - \dots - dy_q^2$$
in coordinates $(x,y)$. Let $\Xi$ be the null cone (using the Euclidian metrics in $x$ resp. $y$)
$|x| = |y|, x,y \neq 0$, and $M \simeq S^{p-1} \times S^{q-1}$ be the submanifold $|x| = |y| = 1$ with
the metric induced from $ds^2$. $O(p,q)$ acts on $\R^{p,q}$ preserving the metric 
and $\Xi$ as well as the lines in $\Xi$; this induced an action on $M$ and one checks
it is a conformal action. Calculating the scalar curvatures of $M$ as well as of $S^{p-1}$ and
$S^{q-1}$ one gets the slightly surprising
\begin{lemma} The Yamabe operators are related by
$$Y_M = Y_{S^{p-1}} - Y_ {S^{q-1}}.$$
\end{lemma}
Thus it is easy to solve the Yamabe equation on $M$ in terms of spherical harmonics
${\mathcal H}^a(\R^p) \otimes {\mathcal H}^b(\R^q), a, b \in \N$. It amounts to
$$(a + \frac{p-2}{2})^2 - (b + \frac{q-2}{2})^2 = 0.$$
\begin{remark} Thus we see that for $p, q > 2$, $p+q$ must be even for such
solutions to exist.
\end{remark}
We denote by $(\pi^{p,q}, V^{p,q})$ the corresponding representation in the kernel
of $Y = Y_M$ with the Hilbert space $V^{p,q}$ described in the following, see \cite{bin} and \cite{Koba1} 
\begin{thm}let $p, q > 2, p+q \in 2\N$, then $(\pi^{p,q}, V^{p,q})$ is a
unitary irreducible representation of $O(p,q)$ with $K$-types
$$\oplus_{a,b \in \N, a + \frac{p}{2} = b + \frac{q}{2}}
{\mathcal H}^a(\R^p) \otimes {\mathcal H}^b(\R^q)$$
and invariant unitary norm given on the spherical harmonics $F_ {a,b}$ by
$$||F_ {a,b}||^2 = (a + \frac{p-2}{2}) ||F_{a,b}||_{L^2(M)}^2.$$
\end{thm}
This representation is minimal in the sense that its Gelfand-Kirillov dimension is
minimal, it is attached to the minimal nilpotent orbit in the Lie algebra $o(p,q)$,
and its annihilator is the so-called Joseph ideal. One may think of it as an analogue of
the metaplectic representation of the double cover of $Sp(n, \R)$ (strictly speaking the
even part of the metaplectic).

\subsection{Branching laws for the minimal $O(p,q)$ representation}
This model of the minimal representation of $G = O(p,q)$ is convenient for studying its restriction to
symmetric subgroups such as $G' = O(p,q') \times O(q''), q = q' + q''$, in particular for
finding the explicit branching law. The idea is to see a product of a hyperboloid and a sphere
as an open orbit of $G'$ on $M$, and to solve the Yamabe equation in the natural coordinates
for $G'$, conformally related to the $M$ coordinates. Thus we shall separate the variables
in the Yamabe equation as those in a hyperboloid and those in a sphere.

First we shall see a remarkable family of unitary irreducible representations $\pi_{+, \lambda}^{p,q}$
of $O(p,q)$ attached to minimal elliptic orbits in the Lie algebra $o(p,q)$; they are realized in spaces
of solutions to $O(p,q)$-invariant differential equations on the hyperboloid
$$X(p,q) \simeq O(p,q)/O(p-1,q)$$
viewed as $\{(x,y)  \in \R^{p,q}| |x|^2 - |y|^2 = 1 \}$ with the induced metric and 
corresponding measure. The most prominent representations
are the {\it discrete series of the hyperboloid}, namely those that arise as irreducible left-invariant
(non-zero) subspaces of $L^2(X(p,q))$, and they are given as solutions to the Yamabe
equation $Y_ {X(p,q)} f = (-\lambda^2 + \frac{1}{4} ) f$ with $\lambda \in \Z + \frac{p+q}{2}$
and $\lambda > 0.$ Actually, the  $(\mathfrak g, K)$-modules allow a slight extension of the
parameter (see \cite{Koba2}, and in particular fact 5.4 for the following details).
\begin{prop}Let $p > 1, q \neq 0$, and $\lambda \in \Z + \frac{p+q}{2}$ and $\lambda > -1.$
Then there is  a unique irreducible non-zero unitarizable (Harish-Chandra) $(\mathfrak g, K)$-module
$(\pi_{+, \lambda}^{p,q})_ K$ consisting of $K$-finite solutions to    
$Y_ {X(p,q)} f = (-\lambda^2 + \frac{1}{4} ) f$. For $\lambda > 0$ this is a
discrete series for the hyperboloid. Furthermore, the $K$-types are given explicitly as 
the direct sum of spherical harmonics
${\mathcal H}^m(\R^p) \otimes {\mathcal H}^n(\R^q)$
with $m,n \in \N, m-n \geq b, m-n \equiv b \, {\rm mod}\, 2$, where $b = \lambda - \frac{p}{2} + \frac{q}{2} +1.$
\end{prop}

There is a similar family $\pi_{-, \lambda}^{p,q}$ that we shall encounter later; it has
similar $K$-types, this time with $m-n \leq b^-, m-n \equiv b^-\,{\rm mod}\, 2$ and
$b^- = -\lambda + \frac{q}{2} - \frac{p}{2} -1.$
\begin{remark} It is a rather subtle point to make these representations unitary for the
exceptional parameters; and they have other characterizations such as cohomologically
induced from $L = SO(2) \times O(p-2,q)$, or as a quotient of a representation induced from
a maximal real parabolic subgroup of $O(p,q)$ whose nilradical is isomorphic to $\R^{p-1, q-1}.$
This parabolic also plays a role in the model of the minimal representation of $O(p,q)$ given in \cite{Koba3}.
\end{remark}
Now we can solve the Yamabe equation on $X(p,q)$ using the
\begin{lemma} There is a conformal embedding
$$X(p,q') \times S^{q''-1
} \subset M \simeq S^{p-1} \times S^{q-1}$$
and the Yamabe operators satisfy
$$Y_ {X(p,q') \times S^{q'' -1}} = Y_{X(p,q')} - Y_{S^{q'' -1}}.$$
\end{lemma}
\begin{proof}Use coordinates, generalizing to a product of two hyperboloids, $p'+p''= p,
q'+q''= q$,
$$X(p', q') \times X(p'',q'') = \{ ((x',y'), (y'',x''))| |x'|^2 - |y'|^2 = |y''|^2 - |x''|^2 = 1 \}.$$
Now we map this into $\Xi$ by $((x',y'), (y'',x'')) \mapsto (x',x'',y',y'')$ and further into $M$
by $$((x',y'), (y'',x'')) \mapsto (\frac{(x',x'')}{|x|}, \frac{(y',y'')}{|y|})$$
with an open image, dense if $p'' = 0.$ This map is conformal with conformal factor
$|x|^{-1} = |y|^{-1}, x = (x',x''), y = (y',y''),$ with respect to the induced metrics from $\R^{p,q}$.
The statement about the Yamabe operators is analogous to the one for the product of spheres,
and also true here in the generalized setting of two hyperboloids (we shall use that below
considering the discrete spectrum in the branching to $G' = O(p',q') \times O(p'',q'')$, but first we
let $p'' = 0$).
\end{proof}
\begin{thm}The restriction of the minimal representation $(\pi^{p,q}, V^{p,q})$ of the
group $G = O(p,q)$ to the symmetric subgroup $G' = O(p,q') \times O(q''), q = q' + q'',$
is the direct Hilbert  space sum 
$$\oplus_{l = 0}^{\infty} \pi_{+, l + \frac{q''}{2} -1}^{p,q'} \otimes {\mathcal H}^l(\R^{q''}).$$
\end{thm}
\begin{proof} Solving the Yamabe equation amounts to the condition on $\pi_{+,\lambda}^{p,q'}
\otimes {\mathcal H}^l(\R^{q''})$ as follows
$$\frac{1}{4} - \lambda^2 + (l + \frac{q''-2}{2})^2 -\frac{1}{4} = 0$$
so we get that $\lambda = l + \frac{q''}{2} -1$ which must be in $\Z + \frac{p+q'}{2}$; but this is
satisfied since $p+q'-q'' = p+q-2q''$ is even. Also we need the condition $\lambda > -1$ which is also
true; note that for $q''\geq 3$ all these elliptic representations have $\lambda > 0$ and hence they
are all discrete series for the hyperboloid. That the solutions we have found  are all in the Hilbert space,
and that the direct sum exhausts the Hilbert space takes some more arguments, see \cite{Koba2}
Theorem 7.1.
\end{proof}
\begin{remark} We can also write the corresponding unitary strucure, in effect the Plancherel formula,
namely if we write $F = \Sigma_{l = 0}^{\infty} F_l^{(1)} F_l^{(2)}$ as in the Theorem, then
$$||F||^2 = \Sigma_{l=0}^{\infty} (l + \frac{q''}{2} - 1) ||F_l^{(1)}||_{L^2(X(p,q'))}^2 
||F_l^{(2)}||_{L^2(S^{q'' - 1})}^2$$
restricting to $q'' \geq 3.$ Note the similarity with the case of a product of spheres.
\end{remark}
Already the case of $q'' = 1$ in the above Theorem is interesting; here there are only two terms
in the sum, since spherical harmonics in one dimension are the constants and the function $x$, i.e
only $l = 0, 1.$ This means that $V^{p,q}_K$, the $K$ finite elements, are realized inside
$C^{\infty}(X(p,q-1))$, namely (here the parameter $\lambda = \pm 1/2$) as solutions to the
Yamabe equation $Y_{X(p,q-1)} f = 0.$
\begin{remark} There are some delicate points to this phenomenon, that the minimal representation
$(\pi^{p,q}, V^{p,q})$ can be realized on the hyperboloid $X(p,q-1)$; note that in a similar way
we may realize the minimal representation $(\pi^{p+1,q-1}, V^{p+1,q-1})$ on the same
hyperboloid $X(p,q-1)$. Indeed we have \cite{Koba3}, Theorem 7.2.2, a non-split exact sequence
of Harish-Chandra modules for $O(p,q-1)$
$$0 \to (\pi^{p,q})_K \to ({\rm Ker}Y_{X(p,q-1)})_K \to (\pi^{p+1,q-1})_K \to 0,$$
which in turn is based on the non-split exact sequences
$$0 \to (\pi^{p,q-1}_{+,-\frac{1}{2}})_K \to C^{\infty}_{\frac{1}{2}, \delta}(X(p,q-1))_K
\to (\pi^{p,q-1}_{-,\frac{1}{2}})_K$$

$$0 \to (\pi^{p,q-1}_{+,\frac{1}{2}})_K \to C^{\infty}_{\frac{1}{2}, -\delta}(X(p,q-1))_K
\to (\pi^{p,q-1}_{-,-\frac{1}{2}})_K$$
where the eigenspaces $C^{\infty}_{\lambda, \epsilon}(X(p,q-1))$ are defined as
$$\{f \in C^{\infty}(X(p,q-1))| Y_{X(p,q-1)}f =
(-\lambda^2 + \frac{1}{4})f, \, f(-z) = \epsilon f(z) \}$$
and the sign character $\delta = (-1)^{(p-q)/2}.$
\end{remark}

As indicated above, there is a similar branching law for the discrete spectrum in restricting
from $G= O(p,q)$ to the other symmetric subgroups $G'$ without a compact factor; we shall use the
following notation for representations $\pi_{-,\lambda}^{p,q}$ of $G$ realized in function spaces
on another hyperboloid $O(p,q)/O(p,q-1)$, namely if we identify $O(p,q)$ with $O(q,p)$, then
$\pi_{-,\lambda}^{p,q}$ corresponds to $\pi_{+,\lambda}^{q,p}.$ We have the same parity conditions
as before on $\lambda$ relative to $p,q$, and we cite the result from \cite{Koba2}, Theorem 9.1
\begin{thm} The restriction of the minimal representation $(\pi^{p,q}, V^{p,q})$ of the
group $G = O(p,q)$ to the symmetric subgroup $G' = O(p',q') \times O(p'',q''), p =  p'+p'', q = q' + q'',$
contains as a discrete spectrum the direct Hilbert  space sum
$$\oplus_{\lambda > 1} \pi_{+, \lambda}^{p',q'} \otimes  \pi_{-, \lambda}^{p'',q''}
\oplus \oplus_{\lambda > 1} \pi_{-, \lambda}^{p',q'} \otimes  \pi_{+, \lambda}^{p'',q''}.$$
Here the sum is over those 
$\lambda > 1$ which satisfy the parity conditions relative to $p',q'$ etc. 
\end{thm}
\begin{proof} Same as before, the Yamabe equation is now separated in the variables corresponding
to the two hyperboloids, so for $\pi_{+, \lambda_+}^{p',q'} \otimes  \pi_{-, \lambda_-}^{p'',q''}$
to satisfy the equation,
we must have $\frac{1}{4} - \lambda_+^2 + \lambda_-^2 - \frac{1}{4} = 0$ which means
$\lambda_+ = \lambda_-.$ The condition $\lambda > 1$ is part of the details in \cite{Koba2}, to see that this
spectrum is contained in the Hilbert space $V^{p,q}.$
\end{proof}

\section{The flat model of the minimal representation of $O(p,q)$}
For the cone in three dimensions: $x^2 + y^2 - z^2 = 0$ (a special case of our $\Xi$)
we know that intersecting with a plane we obtain the conic sections: ellipse, hyperbola, or
parabola, depending on the angle of the plane. This is what we saw for our $\Xi$, namely
the elliptic model, and the hyperbolic model (or some mixture) realizing the minimal
representation. In some sense there is missing the parabolic model; and this is exactly
the model living on the nilpotent radical $N \simeq \R^{p-1, q-1}$ for our maximal real parabolic
subgroup. This will also illustrate further aspects of the representation, in particular the
extension of the discussion of the wave equation, namely the ultrahyperbolic equation.
This will include realizing the Hilbert space as $L^2(C)$, where $C$ arises as the intersection
of the minimal nilpotent orbit in the Lie algebra of $O(p,q)$ with the Lie algebra of $N.$ We shall
also see the role of the Green's function, some Knapp-Stein intertwining operators, and
the fundamental homogeneous distributions on $\R^{p-1, q-1}.$ We fix $p, q \geq 2, p+q > 4$
even.

The Yamabe operator on $\R^{p-1, q-1}$ is in usual coordinates
$$\square = \square_{p-1, q-1} = \frac{\partial^2}{\partial x_1^2} + 
\dots +\frac{\partial^2}{\partial x_{p-1}^2}
-\frac{\partial^2}{\partial x_p^2} 
- \dots -\frac{\partial^2}{\partial x_{p+q-2}^2}$$
and the model we shall now construct lives in the solution space to $\square f = 0.$
Note that for $p = 2$ (or $q = 2$) and restricting to the identity component $SO_0(p,2)$
we are dealing with a unitary highest weight representation, a well-known case, and for
$p = 4$ with the additional physical interpretation of the $K$-types as energy levels of the bound states of
the Hydrogen atom. Note that as in the previous models, the conformal invariance allows us to
study the representation now in these flat coordinates, see \cite{Koba3}. The key object is the
Green's function for $\square$; this is a Schwartz distribution $E_0$ such that by convolution
$$S: C_0^{\infty}(\R^n) \to C^{\infty} (\R^n), \, \varphi \mapsto E_0 \star \varphi$$
has image in the kernel of $\square$, $n = p+q-2.$ We define a Hermitian form on the
image of $S$ by
$$(f_1, f_2)_N = \int_{\R^n} \int_{\R^n} E_0(y-x) \varphi_1(x) \overline{\varphi _2(y)} dx dy$$
where $f_1 = S \varphi_1, f_2 = S \varphi_2.$ This turns out to be well-defined and
positive-definite. The main point is the following, see \cite{Koba3} Theorem 1.4.
\begin{thm}The Hilbert space completion of the image of $S$ with respect to this
Hermitian form carries an irreducible unitary representation of $O(p,q)$ equivalent to
the minimal representation constructed earlier.
\end{thm}
The proof uses in essential ways the explicit form of $E_0$; namely let
$$P(x) = x_1^2 + \dots + x_{p-1}^2 - x_p^2 - \dots  - x_n^2$$
the standard quadratic form in $\R^n$ (which of course also defines the metric). We
define the distribution $(P(x) + i0)^{\lambda}$ as the limit of $(P(x) +iR(x))^{\lambda}$
as a positive definite quadratic form $R(x)$ tends to $0.$ Let
$$E = C(n)e^{i\pi(q-1)/2}(P(x) + i0)^{1 - \frac{n}{2}}$$
with $C(n)$ an explicit real constant.
\begin{prop}$E$ is a fundamental solution of $\square$, namely $\square E = \delta$,
the Dirac delta function, and the Green's function $E_0 = (-E + \overline{E})/2\pi i$
satisfies $\square E_ 0 = 0$, and its Fourier transform is given by
$${\mathcal F}E_0 = \delta(Q).$$
Here $Q(\zeta)$ is the same quadratic form as $P(x)$ in the dual space, and
$\delta (Q)$ can be identified with the invariant measure $d\mu$ supported on
the cone $C = \{ \zeta \in \R^n | Q(\zeta) = 0 \}.$
\end{prop}
This is all quite natural in view of the principle of conformal change in the Yamabe
equation; the parabolic section of the cone $\Xi$ is given by the coordinates $(z', z'') \in
\R^{p-1,q-1} = \R^n$ and the embedding
$$(z', z'') \mapsto (1 - \frac{|z'|^2 - |z''|^2}{4}, z', z'', 1 + \frac{|z'|^2 - |z''|^2}{4}$$
is into $\Xi$, and projects further along rays to a conformal embedding into $M
\simeq S^{p-1} \times S^{q-1}$ (analogous
to what we saw for the product of hyperboloids). Hence the conformal group $O(p,q)$ acts
with singularities (by fractional linear transformations) on the solution space to $\square \varphi = 0$,
and the Fourier transform of solutions are supported on the cone $C$. The convolution operator $S$
with image in this solution space may be identified with a Knapp-Stein $G$-intertwining operator beween
representations induced from the maximal parabolic $P \simeq MAN$ with nilradical $N$ from before, 
$M \simeq O(p-1,q-1)$ and $A$ the dilations with a positive real number. The final insight is to see the Hilbert
space as $(L^2(C), d\mu)$ via ${\mathcal F}(E_0 \star \varphi) = {\mathcal F}(E_0) {\mathcal F}(\varphi).$ We introduce the natural linear map $T: L^2(C) \to {\mathcal S}'(\R^n),
\psi \mapsto \psi d\mu$, to tempered distributions on $\R^n$, viewing the image as a Hilbert
space and $T$ as a unitary operator.
The conclusion is, see \cite{Koba3}, Theorem 4.9.
\begin{thm}${\mathcal F}S(C_0^{\infty}(\R^n)$ is contained in $T(L^2(C))$ with dense image,
and there is a unique (up to a scalar constant) unitary map between $L^2(C)$ and the completion of 
$S(C_0^{\infty}(\R^n).$ This induces a unitary equivalence between unitary irreducible representations $L^2(C)$
and $V^{p,q}$ of $O(p,q).$
\end{thm}
\begin{remark} This model means, using Mackey theory for semidirect product groups, that the
restriction of $\pi^{p,q}$ to $P$ is irreducible (strictly speaking to the opposite group $\overline{P}$).
\end{remark}
Now the natural question becomes to understand the $(\mathfrak g, K)$-module as it is
realized on the cone $C$, in particular what for example are the $K$-types as functions
of $\zeta \in C \subset \R^n.$ Let again $p+q \in 2\Z, p+q>4, p\geq q \geq 2$ and
consider the modified Bessel functions $K_{\nu}(z)$ of the second kind. Then with $|\zeta|$
the Euclidian norm
\begin{thm}$v_0 = |\zeta|^{(3-q)/2} K_ {(q-3)/2} (2|\zeta|)$ is a $K$-finite vector in $L^2(C)$,
in fact the lowest $K$-type, generating under the Fourier transform of the action in $\R^{p-1,q-1}$
of the universal enveloping algebra of $\mathfrak g$ the Harish-Chandra module of the minimal
representation of $O(p,q).$ Hence its completion is isomorphic to $(\pi^{p,q}, V^{p,q}).$
\end{thm}
The action of $\mathfrak g$ is by second-order differential operators on our distributions $\psi d\mu$
in analogy with the metaplectic representation of the symplectic Lie algebra. We finally mention
that there is yet another version of the Hermitian form on the space of solution $\square \varphi = 0$,
involving integration over a hyperplane, in analogy with the use of Cauchy data in the case of the
wave equation, see \cite{Koba3} Theorem 6.2.

\begin{remark} Recall the GJMS operators, which in flat space are $\square^k$, and our hyperboloid
$X(p,q-1)$, also locally conformally flat, and the minimal representation $\pi^{p,q}$ living in the space of solutions to the Yamabe equation $Y_{X(p,q-1)}\varphi = 0,$ and thus via a conformal change also in
a space of solutions to $\square\psi = 0.$ Now in \cite{sch_sch} we find the interesting notion of
Helmholtz numbers. Working on the hyperbola $X(p,q-1)$ and its Laplace operator
$\Delta = \Delta_{X(p,q-1)}$ we are interested in the fundamental solution $u$ to the Helmholtz
equation $(\Delta + \lambda)u = \delta$, for $\delta$ the Dirac distribution at the base point, and a constant $\lambda \in \C.$ There is a natural radial coordinate $r$ (in the case of Riemannian metric it is the geodesic
distance) around the base point, and following Hadamard one looks for a fundamental solution $u$
which near the base point behaves like $r^{2-n}$ (the dimension is here $n = p+q-2$), and
is of the form
$$u = U r^{2-n} + {\mathcal U} {\rm log}r$$
where both $U$ and ${\mathcal U}$ are analytic near $r = 0.$ We say that $\lambda$ is a
{\it Helmholtz number} if there is no log term in the solution to the corresponding Helmholtz equation. For
the case of Lorentz signature metric, where the equation is hyperbolic, this means that it satisfies
the {\it strong Huygens principle}, having a fundamental solution with singular support in the
light-like directions. Now the result in \cite{sch_sch} is among others for other rank one symmetric
spaces, that in our case the Helmholtz numbers are $\lambda_1, \lambda_2, \dots, \lambda_m$,
$m = n/2 - 1$ where $\lambda_l = l(l+n-1), l = -1, -2, \dots, - m.$ Furthermore, the product
$$P(\Delta) = (\Delta + \lambda_1)(\Delta + \lambda_2) \dots (\Delta + \lambda_m)$$
has a log-free fundamental solution as well (and therefore so does also each factor). The last
factor here is the Yamabe operator. 
One may now observe
\begin{prop}The Helmholtz numbers give eigenspaces corresponding to discrete series representations
on $X(p,q-1)$ with an exceptional property, namely by the above singular nature of the eigen-equations; also the
operator $P(\Delta)$ is the GJMS operator of order $2m$ and by a conformal change of coordinates
equivalent to $\square^m$ which also has a fundamental solution of Hadamard type with
no log term.
\end{prop}
The fact that there is such a product formula for the GJMS operators is a general fact about
Einstein manifolds (as we have here), see \cite{case_gover}.
The $\lambda_m$ factor in this product $P(\Delta)$ 
as we said is just the Yamabe operator; 
note the different sign convention we have
here for the Laplacian. It would be interesting to investigate discrete series for other symmetric spaces
to see similar exceptional discrete series for small eigenvalues. Note that the parameter $\lambda$
giving the eigenvalue of the Yamabe operator $Y_{X(p,q-1)}$ is $\lambda = l + \frac{n-1}{2}$, and that
$l = 0, 1, 2, 3, \dots$ yields the rest of the discrete series. Note that these are the eigenvalues on spheres,
and that the same formula gives on the hyperboloids the (finitely many) exceptional eigenvalues,
for negative values of $l$.

Let us finally return to the Yamabe operator with its conformal properties in connection with
these Helmholtz numbers; suppose we have a conformal transformation between two Riemannian
manifolds $\Phi: M \to N, \, \Phi^* g_N = \Omega^2 g_M$ then
$$Y_M \Omega^{(n-2)/2} \Phi^* f = \Omega^{(n+2)/2} \Phi^*Y_Nf$$
and hence
$$(Y_M + \lambda \Omega^2) \Omega^{(n-2)/2} \Phi^* f = \Omega^{(n+2)/2} \Phi^*(Y_N + \lambda)f$$
so that solutions to $(Y_N + \lambda)f = 0$ pull back to solutions of  $(Y_M + \lambda \Omega^2)h = 0.$
We may here think of $N$ as the hyperboloid and $M$ as flat space.
In particular, by the same principle, if $\lambda$ is a Helmholtz number on $N$ (modulo the
constant shift in the Yamabe operator), then 
$(Y_M + \lambda \Omega^2)$ has a log-free fundamental solution. These were found in our flat case
by Stellmacher, see \cite{sch_sch} with the potentials $\Omega^2 = x_j^{-2}.$

Similar considerations for the GJMS operator $P$ of order $2k$ where in the same setting
$$P_M \Omega^{(n-2k)/2} \Phi^* f = \Omega^{(n+2k)/2} \Phi^*P_Nf.$$
What we say here is that by a conformal change of coordinates, again possibly
between two different Riemannian manifolds, also the GJMS operators correspond
to each other; in particular we see a conncection between their fundamental solutions.

\end{remark}

\section{Final remarks on branching laws and symmetry--breaking}

We end with mentioning some additional references  and remarks about the general
themes of branching laws and symmetry-breaking.
This lecture was part of the sequence

\begin{itemize}
\item Introduction to Representation Theory (3h), Birgit Speh
\item Different Aspects of Rankin--Cohen Operators (2h), Michael Pevzner
\item Branching Problems in Representation Theory (3h) Toshiyuki Kobayashi
\item Geometric and Analytic Aspects of Branching Laws (2h) Bent \O{}rsted
\end{itemize}

with the purpose of introducing key concepts and indicating recent advances
in representation theory of Lie groups,
with a special emphasis on branching problems for infinite-dimensional representations of real reductive Lie groups. Below is included a list of references for further reading.
In these lectures was explained some of the basic mathematical theory
of group representations, mostly of reductive Lie groups, and their
structure via the theory of branching problems and intertwining operators. Thus the
aim here is to consider a group $G$ and a subgroup $H$ and a representation
$\pi$ of $G$ in some vector space $V$ (depending on the category, say
Hilbert space for a continuous unitary representation, or a Fr\'e{}chet space for
a smooth representation). Now for a representation $\pi^H$ of $H$ in $V^H$ of a similar nature
we wish to consider the restriction of $\pi$ to $H$ and look for a linear
$H$--equivariant map between them, i.e.
$$T: V \mapsto V^H,\,\,\, \pi^H(h) T = T \pi(h), \,\, (h \in H).$$

Again, depending on the category, we may want $T$ to be say a partial isometry or just
continuous. Such an operator we call {\it symmetry-breaking}, and in the special
case of $H = G$ it is just an {\it intertwining operator}, already an interesting class of
operators. 
See \cite{kob} for an overview. Constructing such operators is closely connected to the problem of
decomposing the restriction of $\pi$ to $H$ as a kind of sum or integral of irreducible
representations of $H$, the so--called {\it branching problem}. 
If $\pi$ is a unitary (always continuous) representation of a Lie group $G$, the
restriction to a one-parameter subgroup $U(t) = \pi(exp(tX))$ is the Fourier transform
of a spectral measure $E$ on the real line, namely
$$ U(t) = \int_{-\infty}^{\infty} e^{itx} dE(x)$$ 
and the generator $d\pi(X) = d/dt|_{t = 0}\pi(exp(tX))$
is a skew-adjoint operator. In some applications in quantum mechanics this spectral measure
and its support has a meaning in terms of observable quantities, for example if the state vector
is $\psi, ||\psi|| = 1$ and $C \subset \R$ a Borel set, then $P\{id\pi(X) \in C\} = (\psi, E(C)\psi)$
is the probability that the observable $id\pi(X)$ will be measured to have a value in $C.$
The corresponding
decomposition of the Hilbert space is thought of as a breaking of the symmetries from $G$.

Now the branching problem for a unitary representation is to find in a similar way the
spectrum of $H$ as explicitly as possible. One is usually interested in the following aspects
of intertwining operators and symmetry-breaking:
\begin{itemize}
\item Intertwining differential operators between $G$-representations
\item Knapp-Stein operators as convolutions with parameter-dependent distributions and their residues
\item Symmetry-breaking operators and their construction as integral and differential operators
\end{itemize}
some examples are, in particular relating to the Dirac operator in \cite{clerc_ors}
and the wave operator with generalizations as in \cite{Koba1}, \cite{Koba2}, and \cite{Koba3}.

Many important differential equations in physics, such as Maxwell's equations, have symmetries, and these reflect the 
relation between the space of solutions and a representation of the group of symmetries, often
a Lie group $G$. Hence as in Galois theory $G$ permutes the solutions of the equation.
Thus one studies, as we began above,
\begin{itemize}
\item natural differential operators in Riemannian and conformal geometry 
\item differential operators on induced representations
\item concrete intertwining integral operators, Knapp--Stein operators, see \cite{knapp} Chapter VII.
\item differential operators as residues of integral operators  
\end{itemize}
with differential forms as a probably less known example, see \cite{fish_ors}. On spheres
we will then have concrete operators that can be analyzed using the analogue of
spherical harmonics as in \cite{bran_ola}. It is an important point (again illustrated above) that conformal
differential geometry and canonical differential operators here are closely connected with
the representation theory of the conformal group of the sphere. Other geometries
arise as well for other groups. For a special case related to the Heisenberg group
see \cite{ku_ors}. Here there is a curious analogue in $CR$ geometry to the wave equation,
namely on the Heisenberg group with the standard left-invariant vector fields $[X,Y] = Z,$
the operator $D = XY + YX$ and its solutions has some analogues with the $O(p,q)$ case above.
We also mention that there are many natural first--order intertwining differential operators as in \cite{ors}.
There is also a natural class of integral symmetry--breaking operators for principal series
representations, see \cite{mo}.
          
For a major background article on symmetry--breaking see \cite{kob}, and
for an important special case with many details \cite{kob_ku}; this is the
 special case of symmetry--breaking differential operators for differential forms, where
the results were also obtained in parallel and independently in \cite{fjs}.
See also \cite{kob_som}.

A final important topic is that of branching laws for discrete series, see
\cite{ors_var1}, \cite{ors_var2}, \cite{ors_var3}, and \cite{var}.


\begin{thebibliography}{12}

\bibitem{alex} Alexakis, Spyros, The decomposition of global conformal invariants. {\it Annals of Mathematics Studies,} {\bf 182}. Princeton University Press, Princeton, NJ, (2012). x+449 pp.

\bibitem{beck} Beckner, William,
Sharp Sobolev inequalities on the sphere and the Moser-Trudinger inequality.
{\it Ann. of Math. (2)} {\bf 138} (1993), no. 1, 213--242.

\bibitem{bin}Binegar, B.; Zierau, R., Unitarization of a singular representation of $SO(p,q)$. {\it Comm. Math. Phys.} {\bf 138} (1991), no. 2, 245--258.

\bibitem{blitz}  Blitz, Samuel; Gover, A. Rod; Waldron, Andrew, Generalized Willmore energies, $Q$-curvatures, extrinsic Paneitz operators, and extrinsic Laplacian powers. {\it Commun. Contemp. Math.}
{\bf 26} (2024), no. 5, Paper No. 2350014, 50 pp.

\bibitem{bra_1} Branson, Thomas P., Sharp inequalities, the functional determinant, and the complementary series. {\it Trans. Amer. Math. Soc.} {\bf 347} (1995), no. 10, 3671--3742.

\bibitem{bra_2} Branson, Thomas P.,  $Q$-curvature, spectral invariants, and representation theory. {\it SIGMA Symmetry Integrability Geom. Methods Appl.} {\bf 3} (2007), Paper 090, 31 pp.

\bibitem{bran_chang}Branson, Thomas P.; Chang, Sun-Yung A.; Yang, Paul C.,
Estimates and extremals for zeta function determinants on four--manifolds.
{\it Comm. Math. Phys.} {\bf 149} (1992), no. 2, 241--262.

\bibitem{bran_ola} Branson T., Olafsson G., \O{}rsted B., Spectrum generating operators and intertwining operators for representations induced from a maximal parabolic subgroup, 
{\it J. Funct. Anal.} {\bf 135} (1996), 163--205.

\bibitem{bran_ors1} Branson, Thomas P.; \O{}rsted, Bent,  Conformal geometry and global invariants. {\it Differential Geom. Appl.} {\bf 1} (1991), no. 3, 279--308.

\bibitem{bran_ors2}  Branson, Thomas P.; \O{}rsted, Bent, Explicit functional determinants in four dimensions. {\it Proc. Amer. Math. Soc.} {\bf 113} (1991), no. 3, 669--682.

\bibitem{bran_ors3} Branson, Thomas P.; \O{}rsted, Bent, Conformal indices of Riemannian manifolds. {\it Compositio Math.} {\bf 60} (1986), no. 3, 261--293.

\bibitem{case_gover} Jeffrey S. Case and A. Rod Gover,
The GJMS operators in geometry, analysis, and physics (September 2025) 
arXiv:2509.16047v1 [math.DG] 

\bibitem{chang}  Chang, Sun-Yung Alice; Yang, Paul C. Prescribing Gaussian curvature on $S^2$. 
{\it Acta Math.} {\bf 159} (1987), no. 3-4, 215--259.

\bibitem{clerc_ors}
Clerc J.L., \O{}rsted B., Conformal covariance for the powers of the Dirac operator, {\it  J. Lie Theory}
{\bf 30} (2020) no. 2, 345--360.

\bibitem{fg1} Fefferman, Charles; Graham, C. Robin, Conformal invariants. The mathematical heritage of \'Elie Cartan (Lyon, 1984). {\it Ast\'erisque} 1985, Num\'ero Hors S\'erie, 95--116.

\bibitem{fg2} Fefferman, Charles; Graham, C. Robin, The ambient metric. {\it Annals of Mathematics Studies}, {\bf 178}. Princeton University Press, Princeton, NJ, 2012. x+113 pp.

\bibitem{fjs}  Fischmann, Matthias; Juhl, Andreas; Somberg, Petr, Conformal symmetry breaking differential operators on differential forms. {\it Mem. Amer. Math. Soc.} {\bf 268} (2020), no. 1304, v+112 pp.

\bibitem{fish_ors} Fischmann, Matthias; \O{}rsted, Bent,  A family of Riesz distributions for differential forms on Euclidian space in {\it Int. Math. Res. Not. IMRN} (2021), {\bf no. 13}, 9746--9768.

\bibitem{frank_lieb} Frank, Rupert L.; Lieb, Elliott H. Sharp constants in several inequalities on the Heisenberg group. {\it Ann. of Math.} (2) {\bf 176} (2012), no. 1, 349--381.

\bibitem{gil}Gilkey, Peter B. Invariance theory, the heat equation, and the Atiyah-Singer index theorem. Second edition. {\it Studies in Advanced Mathematics.} CRC Press, Boca Raton, FL, 1995. x+516 pp.

\bibitem{gover_ors} Gover, A. Rod; \O{}rsted, Bent,
Universal principles for Kazdan-Warner and Pohozaev-Schoen type identities.
{\it Commun. Contemp. Math.} {\bf 15} (2013), no. 4, 1350002, 27 pp.

\bibitem{gjms} Graham, C. Robin; Jenne, Ralph; Mason, Lionel J.; Sparling, George A. J., Conformally invariant powers of the Laplacian. I. Existence. {\it J. London Math. Soc. (2)} {\bf 46} (1992), no. 3, 557--565.

\bibitem{he}Helgason, Sigurdur, Differential geometry and symmetric spaces. {\it Pure and Applied Mathematics}, Vol. XII. Academic Press, New York-London, 1962. xiv+486 pp.

\bibitem{he2} Helgason, Sigurdur, Sophus Lie, the mathematician. The Sophus Lie Memorial Conference (Oslo, 1992), 3--21, Scand. Univ. Press, Oslo, 1994.  

\bibitem{juhl}  Juhl, Andreas, Families of conformally covariant differential operators, $Q$-curvature and holography. {\it Progress in Mathematics}, {\bf 275}. Birkh\"{a}user Verlag, Basel, 2009. xiv+488 pp.

\bibitem{juhl2}  Juhl, Andreas, Explicit formulas for GJMS-operators and $Q$-curvatures. {\it Geom. Funct. Anal.} {\bf 23} (2013), no. 4, 1278--1370. 

\bibitem{juhl_ors} Juhl, Andreas; \O{}rsted, Bent, Residue families, singular Yamabe problems and extrinsic conformal Laplacians. {\it Adv. Math.} {\bf 409} (2022), part A, Paper No. 108634, 158 pp.

\bibitem{knapp} Knapp, Anthony W., Representation theory of semisimple groups. An overview based on examples. {\it Princeton Mathematical Series}, {\bf 36}. Princeton University Press, 
Princeton, NJ, 1986. xviii+774 pp.

\bibitem{kob_som}
Kobayashi T., \O{}rsted B., Somberg P., Sou\v{c}ek V., Branching laws for Verma modules and applications in
parabolic geometry. I, {\it Adv. Math.} {\bf285} (2015), 1796--1852.

\bibitem{kob_ku}    
 Kobayashi T., Kubo T., Pevzner M., Conformal symmetry breaking operators for differential forms on
spheres, Lecture Notes in Math., Vol. 2170, Springer, Singapore, 2016.

\bibitem{kob} 
 Kobayashi T., A program for branching problems in the representation theory of real reductive groups, in
{\it Representations of reductive groups}, {\it Progr. Math.}, {\bf Vol. 312}, Birkh\"a{}user/Springer, Cham, 2015, 277--322.

\bibitem{Koba1} 
 Kobayashi, Toshiyuki; \O{}rsted, Bent: Analysis on the minimal representation of $O(p,q)$. 
I. Realization via conformal geometry. {\it Adv. Math.} {\bf 180} (2003), no. 2, 486--512.

\bibitem{Koba2}
 Kobayashi, Toshiyuki; \O{}rsted, Bent: Analysis on the minimal representation of $O(p,q)$. II. 
Branching laws, in {\it Adv. Math.} {\bf 180} (2003), no. 2, 513--550. 

\bibitem{Koba3}
 Kobayashi, Toshiyuki; \O{}rsted, Bent: Analysis on the minimal representation of $O(p,q)$. 
III. Ultrahyperbolic equations on $\R^{p-1,q-1}$, in  {\it Adv. Math.} {\bf 180} (2003), no. 2, 551--595.


\bibitem{ku_ors} Kubo, Toshihisa; \O{}rsted, Bent, On the space of K--finite solutions to intertwining differential operators, in {\it Representation Theory} {\bf 23} (2019), 213--248.

\bibitem{moe}
M\o{}ller, Niels Martin; \O{}rsted, Bent, Rigidity of conformal functionals on spheres. {\it Int. Math. Res. Not. IMRN} (2014), no. 22, 6302--6339.

\bibitem{moe2} M\o{}ller, Niels Martin, Extremal metrics for spectral functions of Dirac operators in even and odd dimensions. {\it Adv. Math.} {\bf229} (2012), no. 2, 1001--1046.

\bibitem{mo}
M\"o{}llers J., \O{}rsted B., Oshima Y., Knapp--Stein type intertwining operators for symmetric pairs, {\it Adv. Math.}
{\bf 294} (2016), 256--306.

\bibitem{ok} Okikiolu, K., Critical metrics for the determinant of the Laplacian in odd 
dimensions. {\it Ann. of Math.} (2) {\bf 153} (2001), no. 2, 471--531.

\bibitem{osgood}  Osgood, B.; Phillips, R.; Sarnak, P. Extremals of determinants of Laplacians. {\it J. Funct. Anal.} {\bf 80} (1988), no. 1, 148--211.

\bibitem{ro_pa}  Parker, Thomas; Rosenberg, Steven, Invariants of conformal Laplacians. {\it J. Differential Geom.} {\bf 25} (1987), no. 2, 199--222.

\bibitem{sch_sch} Schimming, Rainer; Schlichtkrull, Henrik, Helmholtz operators on harmonic manifolds. {\it 
Acta Math.} {\bf 173} (1994), no. 2, 235--258.

\bibitem{str} Strichartz, Robert S., Linear algebra of curvature tensors and their covariant derivatives. {\it Canad. J. Math.} {\bf 40} (1988), no. 5, 1105--1143.

\bibitem{ors_var1} \O{}rsted, Bent; Vargas, Jorge, Restriction of square integrable representations: discrete spectrum, in {\it Duke Math. J.} {\bf 123} (2004), no. 3, 609--633.

\bibitem{ors_var2} \O{}rsted, Bent; Vargas, Jorge A., Branching problems in reproducing kernel spaces, in
{\it Duke Math. J.} {\bf 169} (2020), no. 18, 3477--3537.

\bibitem{ors_var3}
\O{}rsted, Bent, Vargas, Jorge A.,
Pseudo-dual pairs and branching of Discrete Series
arXiv:2302.14190,  
to appear in {\it Symmetry in Geometry and Analysis -- Festschrift for Toshiyuki Kobayashi.}
M. Pevzner and H. Sekiguchi eds., Progress in Mathematics, Springer 2024.

\bibitem{ors} \O{}rsted, Bent, Generalized gradients and Poisson transforms, in {Global analysis and harmonic analysis} (Marseille-Luminy, 1999), 235--249, {\it S\'e{}min. Congr.}, {\bf 4}, {\it Soc. Math. France}, Paris, 2000.

\bibitem{var} Vargas, J.: Associated symmetric pair and multiplicities of admissible 
restriction of discrete series, in {\it Int.
J. Math.} {\bf 27} 12 (2016).
\end{thebibliography}
\end{document}